\documentclass[11pt]{amsart}

\usepackage{amssymb}
\usepackage{amsfonts}
\usepackage{amsthm}
\usepackage{graphicx}
\usepackage{color}
\usepackage{url}
\usepackage{microtype}
\usepackage{tikz}
\usepackage{tikz-3dplot}
\usetikzlibrary{3d,calc}
\usepackage{bbm}

\usepackage{hyperref}

\usepackage{algorithm}
\usepackage[noend]{algpseudocode}

\makeatletter
\newcommand{\ALOOP}[1]{\ALC@it\algorithmicloop\ #1%
  \begin{ALC@loop}}
\newcommand{\ENDALOOP}{\end{ALC@loop}\ALC@it\algorithmicendloop}

\makeatother

\definecolor{darkblue}{RGB}{0,0,160}
\hypersetup{
colorlinks,%
citecolor=black,%
filecolor=black,%
linkcolor=darkblue,%
urlcolor=darkblue
}

\usepackage[utf8]{inputenc}
\usepackage{todonotes}

\newtheorem{thm}{Theorem}[section]
\newtheorem{lemma}[thm]{Lemma}

\newtheorem{cor}[thm]{Corollary}
\newtheorem{prop}[thm]{Proposition}

\newtheorem{question}[thm]{Question}

\theoremstyle{definition}
\newtheorem{example}[thm]{Example}
\newtheorem{remark}[thm]{Remark}
\newtheorem{defn}[thm]{Definition}

\numberwithin{equation}{section}

\newcommand{\ring}[1]{\ensuremath{\mathbb{#1}}}

\newcommand\NN{\ring{N}}

\newcommand\QQ{\ring{Q}}
\newcommand\RR{\ring{R}}

\newcommand\ZZ{\ring{Z}}

\newcommand\cA{{\mathcal A}}

\newcommand\cC{{\mathcal C}}

\newcommand\cF{{\mathcal F}}
\newcommand\cH{{\mathcal H}}

\newcommand\cM{{\mathcal M}}
\newcommand\cG{{\mathcal G}}
\newcommand\cP{{\mathcal P}}
\newcommand\cR{{\mathcal R}}
\newcommand\cV{{\mathcal V}}

\DeclareMathOperator\rank{rank} 
\DeclareMathOperator\img{img} 
\newcommand{\inD}[1][\relax]{\def\argone{#1}\def\temprelax{\relax}
  \ifx\argone\temprelax\right.\else\,\middle|#1\right.{}\fi}

\newcommand{\rayMat}[3]{A_{#1}(#2,#3)}
\newcommand{\fiber}[2]{\mathcal{F}_{#1,#2}}
\newcommand{\fibergraph}[3]{\mathcal{F}_{#1,#2}\left(#3\right)}

\newcommand{\scfibergraph}[3]{\mathcal{F}^c_{#1,#2}\left(#3\right)}

\newcommand{\ray}[2]{\mathcal{R}_{#2}(#1)}
\newcommand{\longestRay}[2]{\mathcal{R}_{#1,#2}}
\newcommand{\polyray}[3]{\mathcal{R}_{#3,#2}(#1)}
\newcommand{\heatbath}[4]{\mathcal{H}^{#1,#2}_{#3,#4}}
\newcommand{\heatbathmove}[3]{\mathcal{H}^{#1}_{#2,#3}}

\newcommand{\auglen}[2]{\cA_{#1}(#2)}

\newcommand{\graver}[1]{\mathcal{G}_{#1}}

\newcommand{\diam}[1]{\mathrm{diam}(#1)}
\newcommand{\dist}[3]{\mathrm{dist}_{#1}(#2,#3)}
\newcommand{\conv}[1]{\mathrm{conv}_\QQ(#1)}
\newcommand{\spann}[1]{\mathrm{span}_\RR\left\{#1\right\}}
\newcommand{\spannQQ}[1]{\mathrm{span}_\QQ\left\{#1\right\}}
\newcommand\cone[1]{\NN#1}

\begin{document}

\title{Heat-bath random walks with Markov bases}

\author{Caprice Stanley}
\address{NC State University, Raleigh, NC 27695, USA} 
\email{crstanl2@ncsu.edu}

\author{Tobias Windisch}
\address{Otto-von-Guericke Universität\\ Magdeburg, Germany} 
\email{windisch@ovgu.de}

\date{\today}

\makeatletter
  \@namedef{subjclassname@2010}{\textup{2010} Mathematics Subject Classification}
\makeatother




\begin{abstract}
Graphs on lattice points are studied whose edges come from a finite
set of allowed moves of arbitrary length. We show that the diameter of
these graphs on fibers of a fixed integer matrix can be bounded from
above by a constant. We then study the mixing behaviour of heat-bath
random walks on these graphs. We also state explicit conditions on the
set of moves so that the heat-bath random walk, a generalization of
the Glauber dynamics, is an expander in fixed dimension.
\end{abstract}

\maketitle
\setcounter{tocdepth}{1}
\tableofcontents

\section{Introduction}

A \emph{fiber graph} is a graph on the finitely many lattice points
$\cF\subset\ZZ^d$ of a polytope where two lattice points are connected
by an edge if their difference lies in a finite set of allowed moves
$\cM\subset\ZZ^d$. The implicit
structure of these graphs makes them a useful tool to explore the set of lattice
points randomly: At the current lattice point $u\in\cF$, an
element $m\in\pm\cM$ is sampled and the random walk moves along $m$ if
$u+m\in\cF$ and stays at $u$ otherwise. The corresponding Markov chain
is irreducible if the underlying fiber graph is connected and the set
$\cM$ is called a \emph{Markov basis} for $\cF$ in this case. 
This paper investigates the \emph{heat-bath} version of this random
walk: At the current lattice point $u\in\cF$, we sample $m\in\cM$ and
move to a random element in the integer ray $(u+\ZZ\cdot m)\cap\cF$.
The authors of~\cite{Diaconis1998a} discovered that this random
walk can be seen as a discrete version of the \emph{hit-and-run}
algorithm~\cite{lovasz1999,Vempala2005,Lovasz2006} that has been used frequently
to sample from all the points of a polytope -- not only from its
lattice points. The popularity of the continuous version of
the hit-and-run algorithm has not spread to its discrete analog, and
not much is known about its mixing behaviour. One reason is that it is
already challenging to guarantee that all points
in the underlying set $\cF$ can be reached by a random walk that uses
moves from $\cM$, whereas for the continuous version, a random sampling
from the unit sphere suffices. However, in many situations where a
Markov basis is known, the heat-bath random walk is evidently fast.
For instance, it was shown in~\cite{Cryan2002} that the heat-bath
random walk on contingency tables mixes rapidly when the number of
columns is fixed.
To work around the connectedness issue, a \emph{discrete
hit-and-run} algorithm was introduced in~\cite{Baumert2009} for
arbitrary finite sets $\cF\subset\ZZ^d$. 
At each step in this random walk, a subordinate and unrestricted
random walk starts at the current lattice point $u \in \cF$ and uses
the unit vectors to collect a set of proposals $S \subset \ZZ^d$. The
random walk then moves from $u$ to a random point in $S \cap \cF$.

Random walks of the heat-bath type, such as the one presented above, have been
studied recently in~\cite{Dyer2014} in a more general context. In this
paper, we explore the mixing behaviour of heat-bath random walks on
the lattice points of polytopes with Markov bases. Throughout, we
assume that a Markov basis has been found already and refer to the
relevant literature for their
computation~\cite{Sturmfels1996,sullivant2003,hemmecke2005,malkin2007,hara2010,Rauh2014}.
We call the underlying graph of the heat-bath
random walk a \emph{compressed fiber graph}
(Definition~\ref{d:CompressedFibergraph}) and determine in
Section~\ref{s:Diameter} bounds on its graph-diameter. We prove that 
for any $A\in\ZZ^{m\times d}$ with $\ker_\ZZ(A)\cap\NN^d=\{0\}$, the
diameter of compressed fiber graphs on $\{u\in\NN^d:
Au=b\}$ that use a fixed Markov bases
$\cM\subset\ker_\ZZ(A)$ is bounded from above by a constant as $b$ varies
(Theorem~\ref{t:DiameterCompressedFiberGraphs}). In contrast, we 
show that the diameter of conventional fiber graphs grow linearly
under a dilation of the underlying polytope
(Remark~\ref{r:RaysOfMatrices}). This gives
rise to slow mixing results for conventional fiber walks as observed
in~\cite{windisch2015-mixing}.
In Section~\ref{s:HeatBath}, we study in more detail the combinatorial
and analytical structure of the transition matrices of heat-bath random
walks on lattice points and prove upper and lower bounds on their
second largest eigenvalues. We also discuss how the distribution on
the moves $\cM$ affects the speed of convergence
(Example~\ref{ex:SolvingSLEMOpti}).
Theorem~\ref{t:MixingofAugmentingMarkovBases} establishes with
the \emph{canonical path approach} from~\cite{Sinclair1992} an upper
bound on the second largest eigenvalue when the Markov basis is
\emph{augmenting} (Definition~\ref{d:Augmentation}) and the stationary
distribution is uniform. From that, we conclude fast mixing
results for random walks on lattice points in fixed dimension.

\subsection*{Acknowledgements}
CS was partially supported by the US National Science Foundation (DMS
0954865).  TW gratefully acknowledges the support received from the
German National Academic Foundation. 

\subsection*{Conventions and Notation}
The natural numbers are $\NN:=\{0,1,2,\ldots\}$ and for any $N\in\NN$,
$\NN_{> N}:=\{n\in\NN: n> N\}$ and $\NN_{\ge N}:=\{N\}\cup \NN_{> N}$.
For $n\in\NN_{>0}$, let $[n]:=\{1,\ldots,n\}$. Let $\cM\subset\QQ^d$
be a finite set, then $\ZZ\cdot\cM:=\{\lambda m: m\in\cM,
\lambda\in\ZZ\}$ and $\NN\cM$ is the affine semigroup in $\ZZ^d$
generated by $\cM$. For an integer matrix $A\in\ZZ^{m\times d}$ with
columns $a_1,\dots,a_d\in\ZZ^m$, we write $\NN
A:=\NN\{a_1,\ldots,a_d\}$. A
graph is always undirected and can have multiple loops.
The distance of two nodes $u,v$ which are contained in the same
connected component of a graph $G$, i.e. the number of
edges in a shortest path between $u$ and $v$ in $G$, is denoted by
$\dist{G}{u}{v}$. We set $\dist{G}{u}{v}:=\infty$ if $u$
and $v$ are disconnected. A mass
function on a finite set $\Omega$ is a map $f:\Omega\to[0,1]$ such
that $\sum_{\omega\in\Omega}f(\omega)=1$. A mass function $f$ on
$\Omega$ is \emph{positive} if $f(\omega)>0$ for all
$\omega\in\Omega$.  A set $\cF\subset\ZZ^d$ is \emph{normal} if it there
exists a polytope $\cP\subset\QQ^d$ such that $\cP\cap\ZZ^d=\cF$. 

\section{Graphs and statistics}

We first introduce the statistical framework in which this paper
lives and recall important aspects of the interplay between graphs
and statistics. A \emph{random walk} on a graph $G=(V,E)$ is a map
$\cH:V\times V\to[0,1]$ such that for all $v\in V$, $\sum_{u\in
V}\cH(v,u)=1$ and such that $\cH(v,u)=0$ if $\{v,u\}\not\in E$. 
When there is no ambiguity, we represent a random walk as an
$|V|\times|V|$-matrix, for example when it is clear how the elements
of $V$ are ordered. Fix a random walk $\cH$ on $G$.  Then $\cH$ is
\emph{irreducible} if for all $v,u\in V$ there exists $t\in\NN$ such
that $\cH^t(v,u)>0$. The random
walk $\cH$ is \emph{reversible} if there exists a mass function
$\mu:V\to[0,1]$ such that $\mu(u)\cdot \cH(u,v)=\mu(v)\cdot \cH(v,u)$ for
all $u,v\in V$ and \emph{symmetric} if $\cH$ is a symmetric map. A
mass function $\pi:V\to[0,1]$ is a \emph{stationary distribution}
of $\cH$ if $\pi\circ \cH =\pi$. For symmetric random walks, the uniform
distribution on $V$ is always a stationary distribution. 
If $|V|=n$, then we denote the eigenvalues of $\cH$ by
$1=\lambda_1(\cH)\ge\lambda_2(\cH)\ge\dots\ge\lambda_n(\cH)\ge -1$ and we
write $\lambda(\cH):=\max\{\lambda_2(\cH),-\lambda_n(\cH)\}$ for the
\emph{second largest eigenvalue modulus} of $\cH$.  Any irreducible
random walk has a unique stationary
distribution~\cite[Corollary~1.17]{levin2008} and $\lambda(\cH)\in[0,1]$
measures the convergence rate: the smaller $\lambda(\cH)$,
the faster the convergence.

The aim of this paper is to study random walks on lattice points that
use a set of moves. Typically, this is achieved by constructing a
graph on the set of lattice points as follows (compare
to~\cite[Section~1.3]{drton2008} and~\cite[Chapter~5]{Sturmfels1996}).

\begin{defn}\label{d:FiberGraphs}
Let $\cF\subset\ZZ^d$ be a finite set and $\cM\subset\ZZ^d$. The graph
$\cF(\cM)$ is the graph on $\cF$ where two nodes $u,v\in\cF$ are adjacent
if $u-v\in\cM$ or $v-u\in\cM$. 
\end{defn}

A normal set $\cF\subset\ZZ^d$ is finite and satisfies
$\cF=\conv{\cF}\cap\ZZ^d$. A canonical class of normal sets that arise
in many applications, is given by the fibers of an integer matrix:

\begin{defn}\label{d:Fibers}
Let $A\in\ZZ^{m\times d}$ and $b\in\cone{A}$. The set
$\fiber{A}{b}:=\{u\in\NN^d: Au=b\}$ is the \emph{$b$-fiber} of $A$.
The collection of all fibers of $A$ is $\cP_A:=\{\fiber{A}{b}:
b\in\NN A\}$. For $\cM\subset\ker_\ZZ(A)$, the graph
$\fibergraph{A}{b}{\cM}$ is a \emph{fiber graph}.
\end{defn}

Let $\cF,\cM\subset\ZZ^d$ be finite. If the membership in $\cF$ can be
verified efficiently  -- for instance when $\cF$ is given implicitly by
linear equations and inequalities -- then it is possible to explore
$\cF$ randomly using $\cM$ as follows: At a given node $v\in\cF$, a
uniform element $m\in\cM$ is selected. If $v+m\in\cM$, then the random
walk moves along $m$ to $v+m$ and if $v+m\not\in\cM$, the we stay at
$v$. Formally, we obtain the following random walk.

\begin{defn}\label{d:SimpleFiberWalk}
Let $\cF\subset\ZZ^d$ and $\cM\subset\ZZ^d$ be two finite sets. The
\emph{simple walk} is the random walk on $\cF(\cM)$ where the
probability to traverse between to adjacent nodes $u$ and $v$ is
$|\pm\cM|^{-1}$ and the probability to stay at a node $u$ is 
$|\{m\in\pm\cM: u+m\not\in\cF\}|\cdot|\pm\cM|^{-1}$.
\end{defn}

The simple walk is symmetric and hence the uniform distribution is a
stationary distribution (see also~\cite[Section~2]{windisch2015-mixing}). To
ensure convergence, the random walk has to be irreducible, that is,
the underlying graph has to be connected. The following definition is
a slight adaption of the generalized Markov basis as defined
in~\cite[Definition~1]{Rauh2014}.

\begin{defn}\label{d:Markovbasis}
Let $\cP$ be a collection of finite subsets of $\ZZ^d$. A finite set
$\cM\subset\ZZ^d$ is a \emph{Markov basis} of $\cP$, if for all
$\cF\in\cP$, $\cF(\cM)$ is a connected graph. 
\end{defn}

We refer to~\cite[Theorem~3.1]{Diaconis1998a} for a proof that for collections $\cP_A$,
a finite Markov basis always exists and can be computed with tools
from commutative algebra (see also~\cite{hemmecke2005} for more on the
computation of Markov bases).  We now introduce a construction of
graphs on lattice points that also give rise to implementable random
walks, but whose edges have far more reach.

\begin{defn}\label{d:CompressedFibergraph}
Let $\cF\subset\ZZ^d$ and $\cM\subset\ZZ^d$ be finite sets. The
\emph{compression} of the graph $\cF(\cM)$ is the graph
$\cF^c(\cM):=\cF(\ZZ\cdot\cM)$.
\end{defn}

\begin{figure}[htbp]
\centering
\begin{minipage}[b]{0.45\textwidth} 
\begin{minipage}[b]{0.45\textwidth} 
\centering
	\begin{tikzpicture}[xscale=0.5,yscale=0.5]
		\node [fill, circle, inner sep=1.5pt](b2) at (1,3) {};
		\node [fill, circle, inner sep=1.5pt](b2) at (1,4) {};

		\node [fill, circle, inner sep=1.5pt](b2) at (2,2) {};
		\node [fill, circle, inner sep=1.5pt](b2) at (2,3) {};

		\node [fill, circle, inner sep=1.5pt](b2) at (3,1) {};
		\node [fill, circle, inner sep=1.5pt](b2) at (3,2) {};
         
		\node [fill, circle, inner sep=1.5pt](b2) at (4,1) {};


      \foreach \X in {0,1,2,3,4,5}{
         \draw[dotted](\X,0) -- (\X,5);
      }

      \foreach \Y in {0,1,2,3,4,5}{
         \draw[dotted](0,\Y) -- (5,\Y);
      }



      \draw[thick](1,3) --(1,4);
      \draw[thick](2,2) --(2,3);
      \draw[thick](3,1) --(3,2);

       \draw[thick](4,1) -- (1,4);
       \draw[thick](3,1) -- (1,3);

%
%

         \draw [fill=gray, opacity=0.25] (1,3) --(1,4) -- (4,1) --(3,1) --cycle; 

	\end{tikzpicture}
\end{minipage}
\begin{minipage}[b]{0.45\textwidth} 
\centering
	\begin{tikzpicture}[xscale=0.5,yscale=0.5]
		\node [fill, circle, inner sep=1.5pt](b2) at (1,3) {};
		\node [fill, circle, inner sep=1.5pt](b2) at (1,4) {};

		\node [fill, circle, inner sep=1.5pt](b2) at (2,2) {};
		\node [fill, circle, inner sep=1.5pt](b2) at (2,3) {};

		\node [fill, circle, inner sep=1.5pt](b2) at (3,1) {};
		\node [fill, circle, inner sep=1.5pt](b2) at (3,2) {};
         
		\node [fill, circle, inner sep=1.5pt](b2) at (4,1) {};


      \foreach \X in {0,1,2,3,4,5}{
         \draw[dotted](\X,0) -- (\X,5);
      }

      \foreach \Y in {0,1,2,3,4,5}{
         \draw[dotted](0,\Y) -- (5,\Y);
      }



      \draw[thick](1,3) --(1,4);
      \draw[thick](2,2) --(2,3);
      \draw[thick](3,1) --(3,2);

       \draw[thick,bend angle=40, bend right](4,1) to (1,4);
       \draw[thick,bend angle=40, bend right](4,1) to (2,3);
       \draw[thick,bend angle=40, bend right](3,2) to (1,4);
       \draw[thick,bend angle=40, bend left](3,1) to (1,3);

       \draw[thick] (4,1) -- (1,4);
       \draw[thick] (3,1) -- (1,3);

%
%

         \draw [fill=gray, opacity=0.25] (1,3) --(1,4) -- (4,1) --(3,1) --cycle; 

	\end{tikzpicture}
\end{minipage}
\end{minipage}
\begin{minipage}[b]{0.45\textwidth} 
\begin{minipage}[b]{0.45\textwidth} 
\centering
	\begin{tikzpicture}[xscale=0.5,yscale=0.5]
		\node [fill, circle, inner sep=1.5pt](b2) at (1,1) {};
		\node [fill, circle, inner sep=1.5pt](b2) at (1,2) {};
		\node [fill, circle, inner sep=1.5pt](b2) at (1,3) {};
		\node [fill, circle, inner sep=1.5pt](b2) at (1,4) {};

		\node [fill, circle, inner sep=1.5pt](b2) at (2,1) {};
		\node [fill, circle, inner sep=1.5pt](b2) at (2,2) {};
		\node [fill, circle, inner sep=1.5pt](b2) at (2,3) {};

		\node [fill, circle, inner sep=1.5pt](b2) at (3,1) {};
		\node [fill, circle, inner sep=1.5pt](b2) at (3,2) {};
         
		\node [fill, circle, inner sep=1.5pt](b2) at (4,1) {};


      \foreach \X in {0,1,2,3,4,5}{
         \draw[dotted](\X,0) -- (\X,5);
      }

      \foreach \Y in {0,1,2,3,4,5}{
         \draw[dotted](0,\Y) -- (5,\Y);
      }



          \draw[thick](1,1) to (1,4);
          \draw[thick](2,1) to (2,3);
          \draw[thick](3,1) to (3,2);

          \draw[thick](1,1) to (4,1);
          \draw[thick](1,2) to (3,2);
          \draw[thick](1,3) to (2,3);

          \draw [fill=gray, opacity=0.25] (1,1) --(1,4) -- (4,1) --cycle; 

	\end{tikzpicture}
\end{minipage}
\begin{minipage}[b]{0.45\textwidth} 
\centering
	\begin{tikzpicture}[xscale=0.5,yscale=0.5]
		\node [fill, circle, inner sep=1.5pt](b2) at (1,1) {};
		\node [fill, circle, inner sep=1.5pt](b2) at (1,2) {};
		\node [fill, circle, inner sep=1.5pt](b2) at (1,3) {};
		\node [fill, circle, inner sep=1.5pt](b2) at (1,4) {};

		\node [fill, circle, inner sep=1.5pt](b2) at (2,1) {};
		\node [fill, circle, inner sep=1.5pt](b2) at (2,2) {};
		\node [fill, circle, inner sep=1.5pt](b2) at (2,3) {};

		\node [fill, circle, inner sep=1.5pt](b2) at (3,1) {};
		\node [fill, circle, inner sep=1.5pt](b2) at (3,2) {};
         
		\node [fill, circle, inner sep=1.5pt](b2) at (4,1) {};


      \foreach \X in {0,1,2,3,4,5}{
         \draw[dotted](\X,0) -- (\X,5);
      }

      \foreach \Y in {0,1,2,3,4,5}{
         \draw[dotted](0,\Y) -- (5,\Y);
      }



          \draw[thick](1,1) to (1,4);
          \draw[thick](2,1) to (2,3);
          \draw[thick](3,1) to (3,2);

          \draw[thick](1,1) to (4,1);
          \draw[thick](1,2) to (3,2);
          \draw[thick](1,3) to (2,3);

          \draw[thick,bend angle=40, bend left](1,1) to (1,4);
          \draw[thick,bend angle=25, bend left](1,1) to (1,3);
          \draw[thick,bend angle=25, bend left](1,2) to (1,4);

          \draw[thick,bend angle=40, bend right](1,1) to (4,1);
          \draw[thick,bend angle=25, bend right](1,1) to (3,1);
          \draw[thick,bend angle=25, bend right](2,1) to (4,1);

          \draw[thick,bend angle=25, bend right](1,2) to (3,2);
          \draw[thick,bend angle=25, bend left](2,1) to (2,3);

          \draw [fill=gray, opacity=0.25] (1,1) --(1,4) -- (4,1) --cycle; 

	\end{tikzpicture}
   \end{minipage}
   \end{minipage}
	\caption{\label{f:CompressedFibergraph}Compressing graphs.}
\end{figure}
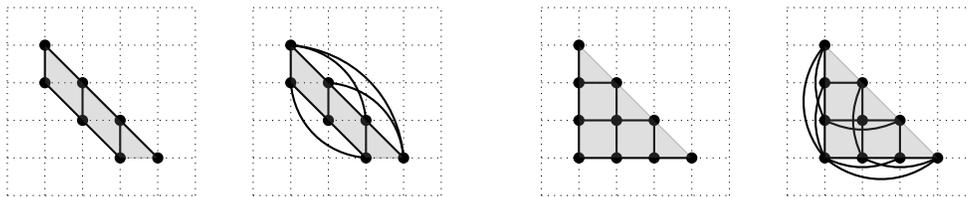

Compressing a graph $\cF(\cM)$ preserves its connectedness: $\cF(\cM)$
is connected if and only if $\cF^c(\cM)$ is connected.

\section{Bounds on the diameter}\label{s:Diameter}

In general knowledge of the diameter of the graph underlying a Markov
chain can provide information about the mixing time. For random walks
on fiber graphs, the chains which we consider, the underlying graph
coincides with the fiber graph. In this section, we determine lower
and upper bounds on the diameter of fiber graphs and their compressed
counterparts.  For a finite set
$\cM\subset\ZZ^d$ and any norm $\|\cdot\|$ on $\RR^d$, let $\|\cM\|:=\max_{m\in\cM}\|m\|$.

\begin{lemma}\label{l:LowerBound}
Let $\cF\subset\ZZ^d$ and $\cM\subset\ZZ^d$ be finite sets, then
\begin{equation*}
\diam{\cF(\cM)}\ge\frac{1}{\|\cM\|}\cdot\max\{\|u-v\|: u,v\in\cF\}.
\end{equation*}
\end{lemma}
\begin{proof}
If $\cF(\cM)$ is not connected, then the statement holds
trivially, so assume that $\cM$ is a Markov basis for $\cF$.  Let $u',v'\in
\cF$ such that $\|u'-v'\|=\max\{\|u-v\|: u,v\in\cF\}$ and let
$m_1,\dots,m_r\in\cM$ so that $u'=v'+\sum_{i=1}^rm_i$ is a path of
minimal length, then $\|u'-v'\|\le r\cdot\|\cM\|$ and the claim follows
from $\diam{\cF(\cM)}\ge\dist{\cF(\cM)}{u'}{v'}=r$.
\end{proof}

\begin{remark}\label{r:LatticeWidth}
Let $\cF\subset\ZZ^d$ be a normal set. For all
$l\in\{-1,0,1\}^d$ and $u,v\in\cF$ we have $(u-v)^Tl\le\|u-v\|_1$ and
thus $\mathrm{width}_l(\cF):=\max\{(u-v)^Tl:
u,v\in\cF\}\le\max\{\|u-v\|_1: u,v\in\cF\}$. 
Suppose that $u',v'\in\cF$ are such that $\|u'-v'\|_1=\max\{\|u-v\|_1:
u,v\in\cF\}$ and let
$l'_i:=\mathrm{sign}(u'_i-v'_i)$ for $i\in[d]$, then
\begin{equation*}
\|u'-v'\|_1= (u'-v')^T\cdot l'\le\mathrm{width}_{l'}(\cF)\le\max\{\|u-v\|_1:
u,v\in\cF\}=\|u'-v'\|_1.
\end{equation*}
The \emph{lattice width} of $\cF$ is
$\mathrm{width}(\cF):=\min_{l\in\ZZ^d}\mathrm{width}_l(\cF)$ and thus
Lemma~\ref{l:LowerBound} gives
\begin{equation*}
\|\cM\|_1\cdot\diam{\cF(\cM)}\ge\mathrm{width}(\cF).
\end{equation*}
\end{remark}

\begin{defn}
Let $\cP$ be a collection of finite subsets of $\ZZ^d$. A
finite set $\cM\subset\ZZ^d$ is
\emph{norm-like} for $\cP$ if there exists a constant $C\in\NN$ such
that for all $\cF\in\cP$ and all $u,v\in\cF$, $\dist{\cF(\cM)}{u}{v}\le
C\cdot\|u-v\|$. The set $\cM$ is $\|\cdot\|$-\emph{norm-reducing} for $\cP$ if for
all $\cF\in\cP$ and all $u,v\in\cF$ there exists $m\in\cM$ such that
$u+m\in\cF$ and $\|u+m-v\|<\|u-v\|$.
\end{defn}

The property of being norm-like does not depend on the norm, whereas
being norm-reducing does. Norm-reducing sets are always norm-like, and
norm-like sets are in turn always Markov bases, but the reverse of both
statements is false in general (Example~\ref{ex:NonNormLikeMB}
and Example~\ref{ex:NonNormReducingNormLike}). For collections $\cP_A$
however, every Markov basis is norm-like
(Proposition~\ref{p:EveryMBIsNormLike}).

\begin{example}\label{ex:NonNormLikeMB}
For any $n\in\NN$, consider the normal set
$\cF_n:=([2]\times[n]\times\{0\})\cup\{(2,n,1)\}$
with the Markov basis $\{(0,1,0),(0,0,1),(-1,0,-1)\}$.  The distance
between $(1,1,0)$ and $(2,1,0)$ in $\cF_n(\cM)$ is $2n$ and thus $\cM$ is
not norm-like for $\{\cF_n:n\in\NN\}$ (see also
Figure~\ref{f:NonNormLikeMB}).
\end{example}

\begin{example}\label{ex:NonNormReducingNormLike}
Let $d\in\NN$ and consider $A:=(1,\dots,1)\in\ZZ^{1\times d}$, then the
set $\cM:=\{e_1-e_i: 2\le i\le d\}$ is a Markov basis for
the collection $\cP_{A}$. However, $\cM$ is not
$\|\cdot\|_p$-norm-reducing for any $d\ge
3$ and any $p\in[1,\infty]$. For
instance, consider $e_2$ and $e_3$ in
$\fibergraph{A}{1}{\cM}$. The only move from $\cM$ that can be applied
on $e_2$ is $e_1-e_2$, but $\|(e_2+e_1-e_2)-e_3)\|_p=\|e_2-e_3\|_p$.
On the other hand,
in the case we cannot find a move that decreases the
$1$-norm of two nodes $u,v\in\fiber{A}{b}$ by $1$, we can
find instead two moves $m_1,m_2\in\cM$ such that
$u+m_1,u+m_1+m_2\in\fiber{A}{b}$ and $\|u+m_1+m_2-v\|=\|u-v\|-2$.
Thus, the graph-distance of any two elements $u$ and $v$ in
$\fibergraph{A}{b}{\cM}$ is at most $\|u-v\|_1$ and hence $\cM$ is
norm-like for $\cP_A$.
\end{example}

\begin{figure}[htbp]
\tdplotsetmaincoords{70}{110}
\begin{tikzpicture}[scale=0.6,tdplot_main_coords]

\def\x{3}
\def\y{8}
\def\z{1}

\draw[thick,->] (0,0,0) -- (\x+0.25,0,0) node[anchor=north east]{};
\draw[thick,->] (0,0,0) -- (0,\y+0.25,0) node[anchor=north west]{};
\draw[thick,->] (0,0,0) -- (0,0,\z) node[anchor=south]{};

\foreach \X in {0,...,\x}
   {\draw[dotted,color=black] (\X,0,0) --(\X,\y,0) node[anchor=north]{};}
\foreach \Y in {0,...,\y} 
   {\draw[dotted,color=black] (0,\Y,0) --(\x,\Y,0) node[anchor=north]{};}

\node[draw,circle,inner sep=0.05cm,fill=black] () at (1,1,0) {};
\node[draw,circle,inner sep=0.05cm,fill=black] () at (1,2,0) {};
\node[draw,circle,inner sep=0.05cm,fill=black] () at (1,3,0) {};
\node[draw,circle,inner sep=0.05cm,fill=black] () at (1,4,0) {};
\node[draw,circle,inner sep=0.05cm,fill=black] () at (1,5,0) {};
\node[draw,circle,inner sep=0.05cm,fill=black] () at (1,6,0) {};
\node[draw,circle,inner sep=0.05cm,fill=black] () at (1,7,0) {};
\node[draw,circle,inner sep=0.05cm,fill=black] () at (2,1,0) {};
\node[draw,circle,inner sep=0.05cm,fill=black] () at (2,2,0) {};
\node[draw,circle,inner sep=0.05cm,fill=black] () at (2,3,0) {};
\node[draw,circle,inner sep=0.05cm,fill=black] () at (2,4,0) {};
\node[draw,circle,inner sep=0.05cm,fill=black] () at (2,5,0) {};
\node[draw,circle,inner sep=0.05cm,fill=black] () at (2,6,0) {};
\node[draw,circle,inner sep=0.05cm,fill=black] () at (2,7,0) {};
\node[draw,circle,inner sep=0.05cm,fill=black] () at (2,7,1) {};

\draw[thick] (1,1,0) -- (1,7,0) ;
\draw[thick] (1,7,0) -- (2,7,1) ;
\draw[thick] (2,7,1) -- (2,7,0) ;
\draw[thick] (2,7,0) -- (2,1,0) ;

\end{tikzpicture}
\caption{The graph from
Example~\ref{ex:NonNormLikeMB}}\label{f:NonNormLikeMB}
\end{figure}
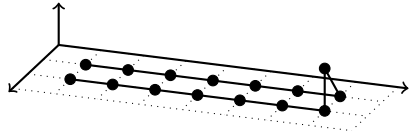

\begin{remark}\label{r:NormLikeUpperDiamBound}
Let $\cP$ be a collection of finite subsets of $\ZZ^d$ and
$\cM\subset\ZZ^d$ be norm-like for $\cP$. It follows
from the definition that there exists a constant $C\in\QQ_{\ge 0}$ such that 
for all $\cF\in\cP$
$$\diam{\cF(\cM)}\le C\cdot\max\{\|u-v\|: u,v\in\cF\}.$$
\end{remark}

The proof of our next results uses the \emph{Graver basis}
$\graver{A}\subset\ZZ^d$ for an integer matrix $A\in\ZZ^{m\times d}$
with $\ker_\ZZ(A)\cap\NN^d=\{0\}$. We refer to
\cite[Chapter~3]{Loera2013} for a precise definition.

\begin{prop}\label{p:EveryMBIsNormLike}
Let $A\in\ZZ^{m\times d}$ with $\ker_\ZZ(A)\cap\NN^d=\{0\}$ and
$\cM\subset\ker_\ZZ(A)$ be a Markov basis of $\cP_A$. Then $\cM$ is
norm-like for $\cP_A$.
\end{prop}
\begin{proof}
Let $\cM$ be a Markov basis for $\cP_A$.
The Graver basis $\graver{A}$ for $A$ is a finite set which is
$\|\cdot\|_1$-norm-reducing for $\cP_A$. 
Thus, define
$C:=\max_{g\in\graver{A}}\diam{\fibergraph{A}{Ag^+}{\cM}}$.
Now, pick $u,v\in\fiber{A}{b}$ arbitrarily and let
$u=v+\sum_{i=1}^r g_i$ be a walk from $u$ to $v$ in
$\fibergraph{A}{b}{\graver{A}}$ of minimal length. Since the Graver
basis is norm-reducing for $\fiber{A}{b}$, there always exists a path
of length at most $\|u-v\|_1$ and hence $r\le\|u-v\|_1$. Every $g_i$
can be replaced by a path in $\fibergraph{A}{Ag_i^+}{\cM}$ of length
at most $C$ and these paths stay in $\fiber{A}{b}$. This gives a path
of length $C\cdot r$, hence $\dist{\fibergraph{A}{b}{\cM}}{u}{v}\le
C\|u-v\|_1$.
\end{proof}

\begin{prop}\label{p:DiamBounds}
Let $\cP\subset\ZZ^d$ be a polytope with $\dim(\cP\cap\ZZ^d)>0$ and let $\cM$ be a Markov
basis for $\cF_i:=(i\cdot\cP)\cap\ZZ^d$ for all $i\in\NN$.
There exists a constant $C'\in\QQ_{>0}$ such that for all $i\in\NN$,
$C'\cdot i\le\diam{\cF_i(\cM)}$. If $\cM$ is norm-like for $\{\cF_i:
i\in\NN\}$, then there exists a constant $C\in\QQ_{>0}$ such that 
$\diam{\cF_i(\cM)}\le C\cdot i$ for all $i\in\NN$.
\end{prop}
\begin{proof}
For the lower bound on the diameter, it suffices to show the existence
of $C'$ such that $C'\cdot i\le\max\{\|u-v\|: u,v\in\cF_i\}$ for all
$i\in\NN$ due to Lemma~\ref{l:LowerBound}. Since
$\dim(\cP\cap\ZZ^d)>0$, we can pick distinct $w,w'\in
\cP\cap\ZZ^d$. For all $i\in\NN$, $i\cdot w,i\cdot w'\in\cF_i$
and hence $i\cdot\|w-w'\|\le\max\{\|u-v\|: u,v\in\cF_i\}$. 

To show the upper bound, assume that $\cM$ is norm-like. It suffices
to show that there exists $C\in\QQ_{\ge 0}$ such that $\max\{\|u-v\|:
u,v\in \cF_i\}\le i\cdot C$ by Remark~\ref{r:NormLikeUpperDiamBound}.
Now, let $v_1,\ldots,v_r\in\QQ^d$ such that
$\cP=\conv{v_1,\ldots,v_r}$ and define $C:=\max\{\|v_s-v_t\|: s\neq
t\}$. 
Since $\cF_i=(i\cdot
\cP)\cap\ZZ^d\subset\conv{iv_1,\ldots,iv_r}$ for all
$i\in\NN$, we have
$\max\{\|u-v\|: u,v\in\cF_i\}\le\max\{\|iv_s-iv_t\|: s\neq
t\}\le C\cdot i$. 
\end{proof}

\begin{remark}\label{r:RaysOfMatrices}
Let $A\in\ZZ^{m\times n}$ with $\ker_\ZZ(A)\cap\NN^d=\{0\}$ and let $\cM$ be a Markov
basis for $\cP_A$. Then $\cM$ is norm-like due to
Proposition~\ref{p:EveryMBIsNormLike} and thus for all $b\in\cone{A}$
there exists $C,C'\in\QQ_{\ge 0}$ such that 
$$i\cdot C'\le\diam{\fibergraph{A}{ib}{\cM}}\le i\cdot C$$ 
for all $i\in\NN$. This generalizes for instance
\cite[Proposition~2.10]{potka2013}
and~\cite[Example~4.7]{windisch2015-mixing}, where linear diameters
on a ray in $\cone{A}$ have been observed.
This also implies that the construction of
expanders from~\cite[Section~4]{windisch2015-mixing} works for every
right-hand side $b\in\cone{A}$.
\end{remark}

\begin{remark}\label{r:Mixing}
Let $A\in\ZZ^{m\times d}$ with $\ker_\ZZ(A)\cap\NN^d=\{0\}$,
$b\in\cone{A}$, and let $\cM$ be a Markov basis
for $\cP_A$. Proposition~\ref{p:DiamBounds} provides a new proof that
the simple walk on $(\fibergraph{A}{ib}{\cM})_{i\in\NN}$ cannot mix
rapidly. The lower bound on the diameter from Proposition~\ref{p:DiamBounds}
implies, in general, the following upper bound on the edge-expansion
(see for example~\cite[Proposition~1.30]{gardam2012}):
\begin{equation*}
h(\fibergraph{A}{i\cdot b}{\cM})\le
|\cM|\left(\exp\left(\frac{\log|\fiber{A}{i\cdot b}|}{D\cdot
i}\right)-1\right).
\end{equation*}
In particular, the edge-expansion cannot be bounded from below by
$\Omega(\frac{1}{p(i)})_{i\in\NN}$ for a polynomial $p\in\QQ[t]$ and
since  $(|\fiber{A}{i\cdot
b}|)_{i\in\NN}\in\mathcal{O}(i^r)_{i\in\NN}$, the simple walk cannot
mix rapidly. In~\cite{windisch2015-mixing}, it was shown that the
edge-expansion can be bounded from above by
$\mathcal{O}(\frac{1}{i})_{i\in\NN}$, which cannot be concluded from
the upper expression. 
\end{remark}

We now turn our attention to the diameter of compressed fiber graphs.
In particular, we want to know for which collections of normal sets is
their diameter bounded. In general, compressing a fiber
graph does not necessarily have an effect on the diameter
(Example~\ref{ex:StairCase}). 

Although a low diameter is a necessary
condition for good mixing, it is not sufficient.  For instance, let
$G_n$ be the disjoint union of two complete graphs $K_n$ connected by
a single edge. Then $\diam{G_n}=3$, but $h(G_n)\le\frac{1}{n}$ implies
that the simple walk does not mix rapidly. 

\begin{example}\label{ex:StairCase}
For any $n\in\NN$, let
$\cF_n:=\{(0,0),(0,1),(1,1),(1,2),\dots,(n,n)\}\subset\ZZ^2$. The
unit vectors $\cM=\{e_1,e_2\}$ are a Markov basis for
$\{\cF_n:n\in\NN\}$. However,
$\cF_n^c(\cM)=\cF_n(\cM)$ and thus $\diam{\cF^c_n(\cM)}=\diam{\cF_n(\cM)}=2n$
is unbounded.
\end{example}

\begin{lemma}\label{l:SignCompatibility}
Let $A\in\ZZ^{m\times d}$ and $z\in\ker_\ZZ(A)$. There exists
$r\in[2d-2]$, distinct elements $g_1,\dots,g_r\in\cG_A$, and
$\lambda_1,\dots,\lambda_r\in\NN_{>0}$ such that
$z=\sum_{i=1}^r\lambda_ig_i$ and $g_i\sqsubseteq z$ for all $i\in[r]$
\end{lemma}
\begin{proof}
This is \cite[Lemma~3.2.3]{Loera2013}, although it only becomes clear
from the original proof of~\cite[Theorem~2.1]{Sebo1990} that the
appearing elements are all distinct.
\end{proof}

\begin{prop}\label{p:DiamCompressedGraverFiberGraphs}
Let $A\in\ZZ^{m\times d}$ and
$\cP:=\left\{\{x\in\ZZ^d: Ax=b, l\le x\le u\}:
l,u\in\ZZ^d,b\in\ZZ^m\right\}$.
Then for all $\cF\in\cP$, 
$\diam{\cF^c(\graver{A})}\le 2d-2$.
\end{prop}
\begin{proof} Let $s,t\in\{x\in\ZZ^d: Ax=b, l\le x\le u\}$, then
$s-t\in\ker_\ZZ(A)$ and thus $s=t+\sum_{i=1}^r\lambda_ig_i$ with $r\le
2d-2$, $\lambda_1,\dots,\lambda_r\in\NN_{>0}$, and distinct
$g_1,\dots,g_r\in\graver{A}$ such that $g_i\sqsubseteq s-t$ according
to Lemma~\ref{l:SignCompatibility}. It's now a
consequence from~\cite[Lemma~3.2.4]{Loera2013} that all intermediate
points $t+\sum_{i=1}^k\lambda_ig_i$ for $k\le r$ are in $\{x\in\ZZ^d:
Ax=b, l\le x\le u\}$.
\end{proof}

\begin{lemma}\label{l:DistOfScaledMoves}
Let $\cF\subset\ZZ^d$ be finite and let 
$\cF_i:=(i\cdot\conv{\cF})\cap\ZZ^d$ for $i\in\NN$. For all $u,v\in\cF$,
$\dist{\cF^c_i(\cM)}{iu}{iv}\le\dist{\cF(\cM)}{u}{v}$ for all
$i\in\NN$.
\end{lemma}
\begin{proof}
The statement is trivially true if $u$ and $v$ are disconnected in
$\cF(\cM)$. Thus, assume the contrary and let $u=v+\sum_{j=1}^k m_j$
with $m_j\in\cM$ be a path in $\cF(\cM)$ of length
$k=\dist{\cF(\cM)}{u}{v}$ and let $i\in\NN$. Clearly, $i\cdot u=i\cdot
v+i\cdot\sum_{j=1}^km_j=i\cdot v+ \sum_{l=1}^ki\cdot m_j$, so it is
left to prove that the elements traversed by this paths are in $\cF_i$.
Let $l\in[k]$, since $v+\sum_{j=1}^lm_j\in\cF$, we have $i\cdot v+\sum_{j=1}^l
i\cdot m_j\in i\cdot \cF\subseteq\cF_i$. Hence, this is a path in
$\cF^c_i(\cM)$ of length $k=\dist{\cF(\cM)}{u}{v}$.
\end{proof}

We are ready to prove that the diameter of compressed fiber graphs
coming from an integer matrix can be bounded for all right-hand sides
simultaneously.

\begin{thm}\label{t:DiameterCompressedFiberGraphs}
Let $A\in\ZZ^{m\times d}$ with $\ker_\ZZ(A)\cap\NN^d=\{0\}$ and let
$\cM$ be a Markov basis for $\cP_A$.  There exists a constant
$C\in\NN$ such that $\diam{\cF^c(\cM)}\le C$ for all $\cF\in\cP_A$.
\end{thm}
\begin{proof}
Our proof relies on basic properties of the Graver basis $\cG_A$ of
$A$. For any $g\in\graver{A}$, let $\cF_g:=\fiber{A}{Ag^+}$
and let
$K:=\max\{\dist{\cF_g(\cM)}{g^+}{g^-}: g\in\graver{A}\}$. We
show that the diameter of any compressed fiber graph of $A$ is
bounded from above by $(2d-2)\cdot K$. 
Let $b\in\cone{A}$ arbitrary and choose elements $u,v\in\fiber{A}{b}$. 
According to Proposition~\ref{p:DiamCompressedGraverFiberGraphs},
there exists $r\in[2d-2]$, 
$g_1,\ldots,g_{r}\in\graver{A}$ and
$\lambda_1,\ldots,\lambda_{r}\in\ZZ$ such that
$u=v+\sum_{i=1}^{r}\lambda_i g_i$, and
$v+\sum_{i=1}^l\lambda_ig_i\in\NN^d$ for all $l\in[r]$. According to
Lemma~\ref{l:DistOfScaledMoves}, for any $i\in[r]$ there are
$m_1^i,\ldots,m_{k_i}^i\in\cM$ and
$\alpha_1,\ldots,\alpha_{k_i}\in\ZZ$ such that
$\lambda_ig_i^+=\lambda_ig_i^-+\sum_{j=1}^{k_i}\alpha_j m_j^i$ is a
path in the compression of $\fiber{A}{A\lambda_ig_i^+}(\cM)$ of length $k_i\le K$. Lifting these
paths for every $i\in[r]$ yields a path
$u=v+\sum_{i=1}^{r}\sum_{j=1}^{k_i}\alpha_jm_j^i$
in $\scfibergraph{A}{b}{\cM}$ of length $r\cdot K\le(2d-2)\cdot K$.
\end{proof}

\section{Heat-bath random walks}\label{s:HeatBath}

In this section, we establish the heat-bath random walk on compressed
fiber graphs. We refer to~\cite{Dyer2014} for a more general
introduction on random walks of heat-bath type. Let
$\cF\subset\ZZ^d$ be finite set. For any $u\in\cF$ and
$m\in\ZZ^d$, the
ray in $\cF$ through $u$ along $m$ is denoted by
$\polyray{u}{m}{\cF}:=(u+m\cdot\ZZ)\cap\cF$. 
Additionally, given a mass function
$\pi:\cF\to[0,1]$, we define
\begin{equation*}
\heatbathmove{\pi}{\cF}{m}(x,y):=\begin{cases}
\frac{\pi(y)}{\pi(\polyray{x}{m}{\cF})}&,\text{ if
}y\in\polyray{x}{m}{\cF}\\
0&,\text{ otherwise}
\end{cases}
\end{equation*}
for $x,y\in\cF$. For $\cM\subset\ZZ^d$ and a mass function $f:\cM\to[0,1]$, the
\emph{heat-bath random walk} is 
\begin{equation}\label{equ:HeatBath}
\heatbath{\pi}{f}{\cF}{\cM}=\sum_{m\in\cM}f(m)\cdot\heatbathmove{\pi}{\cF}{m}.
\end{equation}
The underlying graph of the heat-bath random walk is the compression
$\cF^c(\cM)$ and in this section, we assume throughout that for all
$m\in\cM$ and $\lambda\in\ZZ\setminus\{-1,1\}$, $\lambda\cdot
m\not\in\cM$. Let us first recall the basic properties of this random
walk (compare also to~\cite[Lemma~2.2]{Diaconis1998a}).

\begin{algorithm}[h]
\caption{Heat-bath random walk on compressed fiber
graphs}\label{a:HeatBath}
\begin{algorithmic}[1]
\Require{$\cF\subset\ZZ^d$, $\cM\subset\ZZ^d$,
$v\in\cF$, mass functions
$f:\cM\to[0,1]$ and $\pi:\cF\to[0,1]$, $r\in\NN$}

\Procedure{HeatBath:}{}
\State $v_0:=v$  
\State \texttt{FOR} $s=0$; $s=s+1$, $s<r$  
\State\hspace{\algorithmicindent}Sample $m\in\cM$ according to $f$
\State\hspace{\algorithmicindent}Sample
$v_{s+1}\in\polyray{v_{s}}{m}{\cF}$ according to $\polyray{v_{s}}{m}{\cF}\to[0,1]$,
$y\mapsto\frac{\pi(y)}{\pi(\polyray{v_{s}}{m}{\cF})}$
\State \texttt{RETURN} $v_1,\dots,v_r$
\EndProcedure
\end{algorithmic}
\end{algorithm}

\begin{prop}\label{p:HeatBath}
Let $\cF\subset\ZZ^d$ and $\cM\subset\ZZ^d$ be finite sets. Let
$f:\cM\to[0,1]$ and $\pi:\cF\to(0,1)$ be mass functions. Then
$\heatbath{\pi}{f}{\cF}{\cM}$ is aperiodic, has stationary
distribution $\pi$, is reversible with respect to $\pi$, and all of
its eigenvalues are non-negative. The random walk is irreducible if
and only if $\{m\in\cM: f(m)>0\}$ is a Markov basis for $\cF$.
\end{prop}
\begin{proof}
Since for any $u\in\cF$ and any $m\in\cM$,
$\heatbathmove{\pi}{\cF}{m}(u,u)>0$, there are halting states and thus
$\heatbath{\pi}{f}{\cF}{\cM}$ is aperiodic.  By definition,
$\pi(x)\heatbathmove{\pi}{\cF}{m}(x,y)=\pi(y)\heatbathmove{\pi}{\cF}{m}(y,x)$
and thus $\heatbath{\pi}{f}{\cF}{\cM}$ is reversible with
respect to $\pi$ and $\pi$ is a stationary distribution.  The
statement on the eigenvalues is exactly \cite[Lemma~1.2]{Dyer2014}.
Let $\cM'=\{m\in\cM: f(m)>0\}$ and $f'=f|_{\cM'}$, then
$\heatbath{\pi}{f}{\cF}{\cM}=\heatbath{\pi}{f'}{\cF}{\cM'}$ and thus the
heat-bath random walk is irreducible if and only if $\cM'$ is a Markov
basis for $\cF$.
\end{proof}

\begin{remark}\label{r:ExecutionHeatBath}
Analyzing the speed of convergence of random walks with second largest
eigenvalues does not take the computation time of a single transition
into account. From a computational point of view, the
difference of the simple walk and the heat-bath random walk is 
Step~4 of Algorithm~\ref{a:HeatBath}. However, we argue that Step~4
can be done efficiently in many cases. For instance, a hard
normalizing constant of $\pi$ cancels out. If $\pi$ is the uniform
distribution, then one needs to sample uniformly from
$\polyray{v}{m}{\cF}$ in Step~4, which can be done efficiently. If the
input of Algorithm~\ref{a:HeatBath} is a normal set $\cF=\{u\in\ZZ^d:
Au\le b\}$ that is given in $\mathcal{H}$-representation, then the
length of the ray $\polyray{v}{m}{\cF}$ can be computed with a number
of rounding, division, and comparing operations that is linear in the
number of rows of $A$.  
\end{remark}

There are situations in which the heat-bath random walk provides no
speed-up compared with the simple walk (Example~\ref{ex:NoSpeedup}).
Intuitively, adding more moves to the set of allowed moves should
improve the mixing time of the random walk. In general, however, this
is not true for the heat-bath walk
(Example~\ref{ex:AddingMovesSlowingChain}).

\begin{example}\label{ex:NoSpeedup}
For $n\in\NN$, consider the normal set
\begin{equation*}
\cF_n:=
\left\{
\begin{bmatrix}
0 & 1 & 1 &\cdots& 1\\
1 & 0 & 0 &\cdots& 0\\
\end{bmatrix},
\begin{bmatrix}
1 & 0 & 1 &\cdots& 1 \\
0 & 1 & 0 &\cdots& 0\\
\end{bmatrix},
\ldots,
\begin{bmatrix}
1 & 1 &\cdots & 1 & 0 \\
0 & 0 &\cdots & 0 & 1 \\
\end{bmatrix}\right\}\subset\QQ^{2\times n}.
\end{equation*}
In the language of~\cite[Section~1.1]{drton2008}, $\cF_n$ is precisely
the fiber of the $2 \times n$ independence model where row sums are
$(n-1,1)$ and column sums are $(1,1,\ldots,1)$. The minimal Markov
basis of the independence model, often referred to as the \emph{basic
moves}, is precisely the set $\cM_n:=\{v-u: u,v\in
\cF_n\}\setminus\{0\}$.  In particular, the fiber graph $\cF_n(\cM_n)$ is
the complete graph on $n$ nodes. All rays along basic moves have
length $2$ and thus the transition matrices of the simple random walk
and the heat-bath random walk coincide. There are $n\cdot(n-1)$ many
basic moves and the transition matrix of both random walks is
\begin{equation*}
\frac{1}{n(n-1)}
\begin{bmatrix}
1 & \dots & 1 \\ 
\vdots & & \vdots \\ 
1 & \dots & 1 \\ 
\end{bmatrix}
+
\frac{(n(n-1)-n)}{n(n-1)}\cdot I_n.
\end{equation*}
The second largest eigenvalue is $1-\frac{1}{n-1}$ which implies that
for $n\to\infty$, neither the simple walk nor the heat-bath random walk
are rapidly mixing.
\end{example}

\begin{example}\label{ex:AddingMovesSlowingChain}
Let $\cF=[2]\times[5]\subset\ZZ^2$, $\cM=\{e_1,e_2,2e_1+e_2\}$, and let
$\pi$ be the uniform distribution on $\cF$.  Since $\{e_2,2e_1+e_2\}$ is
not a Markov basis for $\cF$, any mass function $f:\cM\to[0,1]$ must
have $f(e_1)>0$ in order to make the corresponding heat-bath random
walk irreducible. Comparing the second largest eigenvalue modulus of
heat-bath random walks that sample from $\{e_1,e_2\}$ and $\cM$
uniformly, we obtain
\begin{equation*}
\lambda\left(\frac{1}{2}\heatbathmove{\pi}{\cF}{e_1}+\frac{1}{2}\heatbathmove{\pi}{\cF}{e_2}\right)=\frac{1}{2}<
\frac{2}{3}=\lambda\left(\frac{1}{3}\heatbathmove{\pi}{\cF}{e_1}+\frac{1}{3}\heatbathmove{\pi}{\cF}{e_2}+\frac{1}{3}\heatbathmove{\pi}{\cF}{2e_1+e_2}\right).
\end{equation*}
So, adding $2e_1+e_2$ to the set of allowed moves slows the
walk down. This phenomenon does not appear for the simple walk on
$\cF$, where the second largest eigenvalue modulus improves from
$\approx 0.905$ to $\approx 0.888$ when adding the move $2e_1+e_2$ to
the Markov basis. 
\end{example}

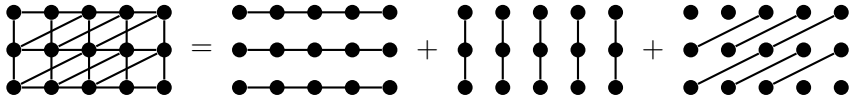
\begin{figure}[h]
\begin{tikzpicture}[xscale=0.5,yscale=0.5]

		\node [fill, circle, inner sep=2pt](D00) at (-6,0) {};
		\node [fill, circle, inner sep=2pt](D10) at (-5,0) {};
		\node [fill, circle, inner sep=2pt](D20) at (-4,0) {};
		\node [fill, circle, inner sep=2pt](D30) at (-3,0) {};
		\node [fill, circle, inner sep=2pt](D40) at (-2,0) {};

		\node [fill, circle, inner sep=2pt](D01) at (-6,1) {};
		\node [fill, circle, inner sep=2pt](D11) at (-5,1) {};
		\node [fill, circle, inner sep=2pt](D21) at (-4,1) {};
		\node [fill, circle, inner sep=2pt](D31) at (-3,1) {};
		\node [fill, circle, inner sep=2pt](D41) at (-2,1) {};

		\node [fill, circle, inner sep=2pt](D02) at (-6,2) {};
		\node [fill, circle, inner sep=2pt](D12) at (-5,2) {};
		\node [fill, circle, inner sep=2pt](D22) at (-4,2) {};
		\node [fill, circle, inner sep=2pt](D32) at (-3,2) {};
		\node [fill, circle, inner sep=2pt](D42) at (-2,2) {};

\draw[thick] (D00) -- (D40);
\draw[thick] (D01) -- (D41);
\draw[thick] (D02) -- (D42);

\draw[thick] (D00) -- (D02);
\draw[thick] (D10) -- (D12);
\draw[thick] (D20) -- (D22);
\draw[thick] (D30) -- (D32);
\draw[thick] (D40) -- (D42);

\draw[thick] (D00) -- (D21);
\draw[thick] (D10) -- (D31);
\draw[thick] (D20) -- (D41);

\draw[thick] (D01) -- (D22);
\draw[thick] (D11) -- (D32);
\draw[thick] (D21) -- (D42);

\node [fill=none, circle, inner sep=2pt](xx) at (-1,1) {$=$};

		\node [fill, circle, inner sep=2pt](A00) at (0,0) {};
		\node [fill, circle, inner sep=2pt](A10) at (1,0) {};
		\node [fill, circle, inner sep=2pt](A20) at (2,0) {};
		\node [fill, circle, inner sep=2pt](A30) at (3,0) {};
		\node [fill, circle, inner sep=2pt](A40) at (4,0) {};

		\node [fill, circle, inner sep=2pt](A01) at (0,1) {};
		\node [fill, circle, inner sep=2pt](A11) at (1,1) {};
		\node [fill, circle, inner sep=2pt](A21) at (2,1) {};
		\node [fill, circle, inner sep=2pt](A31) at (3,1) {};
		\node [fill, circle, inner sep=2pt](A41) at (4,1) {};

		\node [fill, circle, inner sep=2pt](A02) at (0,2) {};
		\node [fill, circle, inner sep=2pt](A12) at (1,2) {};
		\node [fill, circle, inner sep=2pt](A22) at (2,2) {};
		\node [fill, circle, inner sep=2pt](A32) at (3,2) {};
		\node [fill, circle, inner sep=2pt](A42) at (4,2) {};

\draw[thick] (A00) -- (A40);
\draw[thick] (A01) -- (A41);
\draw[thick] (A02) -- (A42);

\node [fill=none, circle, inner sep=2pt](xx) at (5,1) {$+$};

		\node [fill, circle, inner sep=2pt](B00) at (6,0) {};
		\node [fill, circle, inner sep=2pt](B10) at (7,0) {};
		\node [fill, circle, inner sep=2pt](B20) at (8,0) {};
		\node [fill, circle, inner sep=2pt](B30) at (9,0) {};
		\node [fill, circle, inner sep=2pt](B40) at (10,0) {};

		\node [fill, circle, inner sep=2pt](B01) at (6,1) {};
		\node [fill, circle, inner sep=2pt](B11) at (7,1) {};
		\node [fill, circle, inner sep=2pt](B21) at (8,1) {};
		\node [fill, circle, inner sep=2pt](B31) at (9,1) {};
		\node [fill, circle, inner sep=2pt](B41) at (10,1) {};

		\node [fill, circle, inner sep=2pt](B02) at (6,2) {};
		\node [fill, circle, inner sep=2pt](B12) at (7,2) {};
		\node [fill, circle, inner sep=2pt](B22) at (8,2) {};
		\node [fill, circle, inner sep=2pt](B32) at (9,2) {};
		\node [fill, circle, inner sep=2pt](B42) at (10,2) {};

\draw[thick] (B00) -- (B02);
\draw[thick] (B10) -- (B12);
\draw[thick] (B20) -- (B22);
\draw[thick] (B30) -- (B32);
\draw[thick] (B40) -- (B42);

\node [fill=none, circle, inner sep=2pt](xx) at (11,1) {$+$};

		\node [fill, circle, inner sep=2pt](C00) at (12,0) {};
		\node [fill, circle, inner sep=2pt](C10) at (13,0) {};
		\node [fill, circle, inner sep=2pt](C20) at (14,0) {};
		\node [fill, circle, inner sep=2pt](C30) at (15,0) {};
		\node [fill, circle, inner sep=2pt](C40) at (16,0) {};

		\node [fill, circle, inner sep=2pt](C01) at (12,1) {};
		\node [fill, circle, inner sep=2pt](C11) at (13,1) {};
		\node [fill, circle, inner sep=2pt](C21) at (14,1) {};
		\node [fill, circle, inner sep=2pt](C31) at (15,1) {};
		\node [fill, circle, inner sep=2pt](C41) at (16,1) {};

		\node [fill, circle, inner sep=2pt](C02) at (12,2) {};
		\node [fill, circle, inner sep=2pt](C12) at (13,2) {};
		\node [fill, circle, inner sep=2pt](C22) at (14,2) {};
		\node [fill, circle, inner sep=2pt](C32) at (15,2) {};
		\node [fill, circle, inner sep=2pt](C42) at (16,2) {};

\draw[thick] (C00) -- (C21);
\draw[thick] (C10) -- (C31);
\draw[thick] (C20) -- (C41);
\draw[thick] (C01) -- (C22);
\draw[thick] (C11) -- (C32);
\draw[thick] (C21) -- (C42);

	\end{tikzpicture}
   \caption{Decomposition of the graph in
   Example~\ref{ex:AddingMovesSlowingChain}}\label{f:AddingMovesSlowingChain}
\end{figure}

\begin{remark}\label{r:GlauberDynamics}
Let $\cF\subset\ZZ^d$ be finite and $\cM=\{m_1,\dots,m_d\}\subset\ZZ^d$ be a
linearly independent Markov basis of $\cF$. If the moves are selected
uniformly, then the heat-bath random walk on $\cF$ coincides with the
\emph{Glauber dynamics} on $\cF$. To see it, choose $u\in\cF$ and let
\begin{equation*}
\cF':=\{\lambda\in\ZZ^d: u+\lambda_1m_1+\dots+\lambda_dm_d\in\cF\}.
\end{equation*}
It is easy to check that $\cF'$ is unique up to translation and
depends only on $\cF$, $u$, and $\cM$. Since the vectors in $\cM$ are linearly
independent, every element of $\cF$ can be represented by a
unique choice of coefficients in $\cF'$. Thus, the heat-bath random walk
on $\cF$ using $\cM$ is equivalent to the heat-bath random walk on
on $\cF'$ using the unit vectors as moves. For any unit vector
$e_i\in\ZZ^d$, the ray through an element $v\in\cF'$ is $\{w\in\cF:
w_j=v_j \forall j\neq i\}$ and this is precisely the form desired in
the Glauber dynamics~\cite[Section~3.3.2]{levin2008}.
\end{remark}

For the remainder of this section, we primarily focus on heat-bath
random walks $\heatbath{\pi}{f}{\cF}{\cM}$ that converge to the
uniform distribution $\pi$ on a finite, but not necessarily normal, set $\cF$. We particularly
aim for bounds on its second largest eigenvalue by
making use of the decomposition from equation~\ref{equ:HeatBath}. Our
first observations consider its summands $\heatbathmove{\pi}{\cF}{m}$
that can be well understood analytically
(Proposition~\ref{p:MoveMatrix}) and combinatorially
(Proposition~\ref{p:IntersectionTheorem}). 

\begin{prop}\label{p:MoveMatrix}
Let $\cF\subset\ZZ^d$ be a finite set, $m\in\ZZ^d$, and $\pi:\cF\to[0,1]$
be the uniform distribution. Let $\cR_1,\dots,\cR_k$ be the disjoint
rays through $\cF$ along $m$. Then
\begin{enumerate}
\item $\heatbathmove{\pi}{\cF}{m}$ is symmetric and idempotent.
\item
$\img(\heatbathmove{\pi}{\cF}{m})=\spann{\sum_{x\in\cR_1}e_x,\sum_{x\in\cR_2}e_x,\dots,\sum_{x\in\cR_k}e_x}$.
\item $\ker(\heatbathmove{\pi}{\cF}{m})=\bigoplus_{i=1}^k\spann{e_x-e_y:
x,y\in\cR_i, x\neq y}$.
\item $\rank(\heatbathmove{\pi}{\cF}{m})=k$ and
$\dim\ker(\heatbathmove{\pi}{\cF}{m})=|\cF|-k$.
\item The spectrum of $\heatbathmove{\pi}{\cF}{m}$ is $\{0,1\}$.
\end{enumerate}
\end{prop}
\begin{proof}
Symmetry of $\heatbathmove{\pi}{\cF}{m}$ follows from the
definition. By assumption, $\cF$ is the disjoint union of
$\cR_1,\dots,\cR_k$ and hence there exists a permutation matrix
$S$ such that $S\heatbathmove{\pi}{\cF}{m}S^T$ is a block matrix whose
building blocks are the matrices 
\begin{equation*}
\frac{1}{|\cR_i|}\begin{bmatrix}
1 & \dots & 1 \\ 
\vdots & & \vdots \\ 
1 & \dots & 1 \\ 
\end{bmatrix}\in\QQ^{|\cR_i|\times|\cR_i|}.
\end{equation*}
Thus, $\heatbathmove{\pi}{\cF}{m}$ is
idempotent and the rank of $\heatbathmove{\pi}{\cF}{m}$ is $k$. A basis
of its image and its kernel can be read off directly and
idempotent matrices can only have the eigenvalues $0$ and~$1$. 
\end{proof}

\begin{prop}\label{p:IntersectionTheorem}
Let $\cF\subset\ZZ^d$ and $\cM\subset\ZZ^d$ be finite sets,
$\pi:\cF\to[0,1]$ be the uniform distribution, and let
$V_1,\dots,V_c\subseteq\cF$ be the nodes of the connected components of
$\cF(\cM)$, then
$$\bigcap_{m\in\cM}\img(\heatbathmove{\pi}{\cF}m)=\spann{\sum_{x\in
V_1}e_x,\dots,\sum_{x\in V_c}e_x}.$$
\end{prop}
\begin{proof}
It is clear by Proposition~\ref{p:MoveMatrix} that the set
on the right-hand side is contained in any
$\img(\heatbathmove{\pi}{\cF}m)$ since any $V_i$ decomposes disjointly into rays
along $m\in\cM$.
To show the other inclusion, write $\cM=\{m_1,\dots,m_k\}$ and let for any $i\in[k]$,
$\cR^i_1,\dots,\cR^i_{n_i}$ be the disjoint rays through $\cF$ parallel
to $m_i$. In particular, $\{\cR^i_1,\dots,\cR^i_{n_i}\}$ is a
partition of $\cF$ for any $i\in[k]$. Let $v\in\bigcap_{m\in\cM}\img(\heatbathmove{\pi}{\cF}m)$.
Again by Proposition~\ref{p:MoveMatrix}, there exists for any
$i\in[k]$, $\lambda^i_1,\dots,\lambda^i_{n_i}\in\QQ$ such that
\begin{equation*}
v=\sum_{j=1}^{n_i}\sum_{x\in\cR^i_j}\lambda_j^ie_x.
\end{equation*}
Notice that if two distinct Markov moves $m_i$ and $m_{i'}$
and two indices $j\in[n_i]$ and $j'\in[n_{i'}]$
satisfy $\cR_{j}^i\cap\cR^{i'}_{j'}\neq\emptyset$,
then $\lambda^i_j=\lambda^{i'}_{j'}$. We show that for
any $i\in[k]$ and any $a\in[c]$, $\lambda_{j}^i=\lambda_{j'}^i$ when
$\cR^i_j$ and $\cR^i_{j'}$ are a subset of $V_a$. This implies the
proposition. So take distinct $x,x'\in V_a$
and assume that $x$ and $x'$ lie on different rays of
$m_i$ and let that be $x\in\cR^i_{j}$ and $x'\in\cR^i_{j'}$ with
$j\neq j'$. Since $x$ and $x'$ are in the same connected component
$V_a$ of $\cF(\cM)$, let $y_{i_0},\dots,y_{i_r}\in\cF$ be the nodes on a
minimal path in $\cF^c(\cM)$ with $y_{i_0}=x$ and $y_{i_r}=x'$. For any $s\in[r]$,
$y_{i_s}$ and $y_{i_{s-1}}$ are contained in the same ray
$\cR^{k_s}_{t_s}$ coming from a
Markov move $m_{k_s}$. In particular,
$\cR_{k_{s-1}}^{t_{s-1}}\cap\cR^{k_s}_{t_s}\neq\emptyset$ and due to
our observation made above
$\lambda_{j}^i=\lambda_{t_1}^{k_1}=\lambda_{t_2}^{k_2}=\dots=\lambda_{t_r}^{k_r}=\lambda_{j'}^i$
which finishes the proof.
\end{proof}

\begin{defn}\label{d:RayMatrix}
Let $\cF\subset\ZZ^d$ and $\cM\subset\ZZ^d$ be finite sets and
$\cM'\subseteq\cM$. Let $\mathcal{V}$ be the set of
connected components of $\cF(\cM\setminus\cM')$ and
$\cR$ be the set of all rays through $\cF$ along all elements
of $\cM'$. The \emph{ray matrix} of $\cF(\cM)$ along $\cM'$
is
$\rayMat{\cF}{\cM}{\cM'}:=(|R\cap
V|)_{R\in\cR,V\in\cV}\in\NN^{\cR\times\cV}$.
\end{defn}

\begin{example}
Let $\cF =[3]\times [3]$, $\cM=\{e_1, e_2, e_1+e_2 \}$, and $\cM'= \{
e_1, e_2\}$. Then $\cF(\cM\setminus \cM')$ has five connected
components and the ray matrix of $\cF(\cM)$ along
$\cM'$ is 
\[ \rayMat{\cF}{\cM}{\cM'}=
\begin{bmatrix}
1 & 1 & 1 & 0 & 0 \\ 
0 & 1 & 1 & 1 & 0 \\ 
0 & 0 & 1 & 1 & 1 \\
0 & 0 & 1 & 1 & 1\\
0 & 1 & 1 & 1 & 0\\
1 & 1 & 1 & 0 & 0
\end{bmatrix}.
\]
\end{example}

\begin{remark}\label{r:RayMatUnitVecTranspose}
Let $\cF\subset\ZZ^2$, then the rays through $\cF$ along $e_1$ are the
connected components of $\cF(\{e_1,e_2\}\setminus\{e_2\})$ and the rays through $\cF$
along $e_2$ are the connected components of
$\cF(\{e_1,e_2\}\setminus\{e_1\})$, thus
$\rayMat{\cF}{\cM}{e_1}=\rayMat{\cF}{\cM}{e_2}^T$.
\end{remark}

\begin{prop}\label{p:RaysAndCC}
Let $\cF\subset\ZZ^d$ and $\cM\subset\ZZ^d$ be finite sets,
$\pi:\cF\to[0,1]$ be the uniform distribution, and $\cM'\subseteq\cM$.
Then
\begin{equation*}
\ker(\rayMat{\cF}{\cM}{\cM'})\cong
\bigcap_{m\in\cM\setminus\cM'}\img(\heatbathmove{\pi}{\cF}{m})
\cap\bigcap_{m\in\cM'}\ker(\heatbathmove{\pi}{\cF}{m}).
\end{equation*}
\end{prop}
\begin{proof}
Let $V_1,\dots,V_c$ be the connected components of $\cF(\cM\setminus\cM')$ and
$\cR_1,\dots,\cR_r$ be the rays along elements in $\cM'$. 
Let $I:=\bigcap_{m\in\cM\setminus\cM'}\img(\heatbathmove{\pi}{\cF}{m})$ and
$K:=\bigcap_{m\in\cM'}\ker(\heatbathmove{\pi}{\cF}{m})$. By
Proposition~\ref{p:IntersectionTheorem}, any element of $I$ has the
form $v=\sum_{i=1}^c(\lambda_i\sum_{x\in V_i}e_x)$ for
$\lambda_1,\dots,\lambda_c\in\QQ$. Assume additionally that
$v\in\ker(\heatbathmove{\pi}{\cF}{m})$ for $m\in\cM'$ and let
$\cR_{i_1},\dots\cR_{i_j}$ be the rays which belong to $m$,
then for any $k\in[j]$,
$0=\sum_{x\in\cR_{i_k}}v_x=\sum_{j=1}^c\lambda_j|\cR_{i_k}\cap V_j|$. Put 
differently, a vector $\lambda\in\RR^c$ is in the kernel
of $(|\cR_i\cap V_j|)_{i\in[r],j\in[c]}$ if and only if
$\sum_{i=1}^c(\lambda_i\sum_{x\in V_i}e_x)\in I\cap K$.  
\end{proof}

Conditions on the kernel of the ray matrix allow us to give a lower bound on the 
second largest eigenvalue of the heat-bath random walk.

\begin{prop}\label{p:KernelRayMat}
Let $\cF\subset\ZZ^d$ and $\cM\subset\ZZ^d$ be finite sets and
$\pi$ be the uniform distribution.  Let
$\cM'\subseteq\cM$ such that $\ker(\rayMat{\cF}{\cM}{\cM'})\neq\{0\}$, then $\lambda(\heatbath{\pi}{f}{\cF}{\cM})\ge
1-\sum_{m\in\cM'}f(m)$ for any mass function $f:\cM\to[0,1]$.
\end{prop}
\begin{proof}
Using the isomorphism from Proposition~\ref{p:RaysAndCC}, we can
choose a non-zero $v\in\QQ^P$ such that
$\heatbathmove{\pi}{\cF}{m}v=v$ for all $m\in\cM\setminus\cM'$ and
$\heatbathmove{\pi}{\cF}{m}v=0$ for all $m\in\cM'$. In particular
\begin{equation*}
\heatbath{\pi}{f}{\cF}{\cM}v=\sum_{m\in\cM}f(m)\heatbathmove{\pi}{\cF}{m}v=
\sum_{m\in\cM\setminus\cM'}f(m)\heatbathmove{\pi}{\cF}{m}v=\sum_{m\in\cM\setminus\cM}f(m)v.
\end{equation*}
Since $f$ is a mass function, $1-\sum_{m\in\cM'}f(m)$ is an
eigenvalue of $\heatbath{\pi}{f}{\cF}{\cM}$.
\end{proof}

\begin{defn}\label{d:IntersectingRayProperty}
Let $\cF\subset\ZZ^d$ and $m,m'\in\ZZ^d$ not collinear. The pair
$(m,m')$ has the \emph{intersecting
ray property} in $\cF$ if the following holds:
For any pair of rays $\cR_1,\cR_2$ parallel to $m$ and any pair of rays
$\cR_1',\cR_2'$ parallel to $m'$ where both $\cR_1\cap\cR_1'$ and
$\cR_2\cap\cR_2'$ are not empty, then
$\cR_1\cap\cR_2'\neq\emptyset$ implies $\cR_1'\cap\cR_2\neq\emptyset$
and $|\cR_1|\cdot|\cR_1'|^{-1}=|\cR_2|\cdot|\cR_2'|^{-1}$. For a
finite set $\cM\subset\ZZ^d$, the graph $\cF^c(\cM)$ has the
\emph{intersecting ray property} if all $(m,m')$ have the intersecting
ray property in $\cF$.
\end{defn}

\begin{example}
The compressed fiber graph on
$[n_1]\times\cdots\times[n_d]\subset\ZZ^d$ that uses the unit vectors
$\{e_1,\dots,e_d\}$ as moves has the intersecting ray property. On the
other hand, consider $\cF=\{u\in\NN^2: u_1+u_2\le 1\}$ and take the
rays $\cR_1:=\{(0,0),(0,1)\}$ and $\cR_2:=\{(1,0)\}$ that
are parallel to $e_2$ and the rays $\cR_1':=\{(0,1)\}$ and
$\cR_2':=\{(0,0),(1,0)\}$ that are parallel to $e_1$. Then
$\cR_1\cap\cR_1'=\{(1,0)\}$ and $\cR_2\cap\cR_2'=\{(0,1)\}$, but
$\cR_1\cap\cR_2'=\{(0,0)\}\neq\emptyset$ and
$\cR_1'\cap\cR_2=\emptyset$.
\end{example}

\begin{prop}\label{p:CommutingMoveMatrices}
Let $m,m'\in\ZZ^d$ not collinear and $\cF\subset\ZZ^d$ be a finite set. The
matrices $\heatbathmove{\pi}{\cF}{m}$ and $\heatbathmove{\pi}{\cF}{m'}$
commute if and only if $(m,m')$ have the intersecting ray property in
$\cF$.
\end{prop}
\begin{proof}
Let $u_1,u_2\in\cF$. Then 
\begin{equation*}
(\heatbathmove{\pi}{\cF}{m}\cdot\heatbathmove{\pi}{\cF}{m'})_{u_1,u_2}=
\begin{cases}
|\polyray{u_1}{m}{\cF}|^{-1}\cdot|\polyray{u_2}{m'}{\cF}|^{-1},&\text{ if
}\polyray{u_1}{m}{\cF}\cap\polyray{u_2}{m'}{\cF}\neq\emptyset\\
0,&\text{ otherwise}
\end{cases}.
\end{equation*}
Let $\cR_1:=\polyray{u_1}{m}{\cF}$, $\cR_1':=\polyray{u_1}{m'}{\cF}$,
$\cR_2:=\polyray{u_2}{m}{\cF}$, and $\cR_2':=\polyray{u_2}{m'}{\cF}$
Thus,
$(\heatbathmove{\pi}{\cF}{m}\cdot\heatbathmove{\pi}{\cF}{m'})_{u_1,u_2}=(\heatbathmove{\pi}{\cF}{m'}\cdot\heatbathmove{\pi}{\cF}{m})_{u_1,u_2}$.
It is easy to see that the matrices commute if and only if $(m,m')$
have the intersecting ray property.
\end{proof}

\begin{lemma}\label{l:DiagonizableCommutingMatrices}
Let $H_1,\dots,H_n\in\RR^{n\times n}$ be pairwise
commuting matrices. Then any eigenvalue of $\sum_{i=1}^nH_i$ has the
form $\lambda_1+\dots+\lambda_n$ where $\lambda_i$ is an eigenvalue
of~$H_i$.
\end{lemma}
\begin{proof}
This is a straightforward extension of the case $n=2$
in~\cite[Theorem~2.4.8.1]{Horn2013} and relies on the fact that
commuting matrices are simultaneously triangularizable.
\end{proof}

\begin{prop}
Let $\cF\subset\ZZ^d$ and $\cM\subset\ZZ^d$ be finite sets and suppose
there exists $m\in\cM$ such that $(m,m')$ has the intersecting ray
property in $\cF$ for all $m'\in\cM':=\cM\setminus\{m\}$. Let
$\cV_1,\dots,\cV_c$ be the connected components of
$\cF(\cM')$, $\pi_i:\cV_i\to[0,1]$ the uniform distribution, and
$f'=(1-f(m))^{-1}\cdot f|_{\cM'}$, then
\begin{equation*}
\lambda(\heatbath{\pi}{f}{\cF}{\cM})\le
f(m)+(1-f(m))\cdot\max\{\lambda(\heatbath{\pi_i}{f'}{\cV_i}{\cM'}):i\in
[c]\}.
\end{equation*}
\end{prop}
\begin{proof}
Let $\cH:=\heatbath{\pi}{f'}{\cF}{\cM'}$
be the heat-bath random walk on $\cF(\cM)$ that samples moves from
$\cM'$ according to $f'$, then
$\heatbath{\pi}{f}{\cF}{\cM}=f(m)\cdot\heatbathmove{\pi}{\cF}{m}+(1-f(m))\cdot
\cH$.
By assumption, all pairs $(m,m')$ with $m'\in\cM'$ have the intersecting ray
property and thus the matrices $\heatbathmove{\pi}{\cF}{m}$ and
$\cH$ commute according to
Proposition~\ref{p:CommutingMoveMatrices}. The eigenvalues of all
involved matrices are non-negative and thus
Lemma~\ref{l:DiagonizableCommutingMatrices} implies that the second largest
eigenvalue of $\heatbath{\pi}{f}{\cF}{\cM}$ has the form
$\lambda+\lambda'$ where $\lambda\in\{0,f(m)\}$ by
Proposition~\ref{p:MoveMatrix} and where $\lambda'$
is an eigenvalue of $(1-f(m))\cdot\cH$.
The matrix $\cH$ is a block matrix
whose building blocks are the matrices
$\heatbath{\pi}{f'}{\cV_i}{\cM'}=\heatbath{\pi_i}{f'}{\cV_i}{\cM'}$
and thus the statement follows.
\end{proof}

\begin{prop}\label{p:UpperSLEMBound}
Let $\cF\subset\ZZ^d$ and $\cM\subset\ZZ^k$ be finite sets. If
$\cF(\cM)$ has the intersecting ray
property, then $\lambda(\heatbath{\pi}{f}{\cF}{\cM})\le 1-\min(f)$.
\end{prop}
\begin{proof}
Let $\cM=\{m_1,\dots,m_k\}$. The intersecting ray property and
Proposition~\ref{p:CommutingMoveMatrices} give that
the matrices
$f(m_1)\cdot\heatbathmove{\pi}{\cF}{m_i},\dots,f(m_k)\cdot\heatbathmove{\pi}{\cF}{m_k}$
commute pairwise. According to Proposition~\ref{p:MoveMatrix}, the
eigenvalues of
$f(m_i)\cdot\heatbathmove{\pi}{\cF}{m_i}$ are $\{0,f(m_i)\}$.
Lemma~\ref{l:DiagonizableCommutingMatrices} gives that the second
largest eigenvalue of $\heatbath{\pi}{f}{\cF}{\cM}$, which equals the
second largest eigenvalue modulus since all of its eigenvalues are
non-negative, fulfills
$\lambda(\heatbath{\pi}{f}{\cF}{\cM})=\sum_{i\in
I}f(m_i)$ for a subset $I\subseteq[k]$. Since
$\lambda(\heatbath{\pi}{f}{\cF}{\cM})<1$ and $\sum_{i=1}^kf(m_i)=1$, we
have $I\neq[k]$ and the claim follows.
\end{proof}

\begin{prop}\label{p:HeatBathOnHyperrectangle}
Let $n_1,\dots,n_d\in\NN_{>1}$,
$\cF=[n_1]\times\cdots\times[n_d]$, and
$\cM=\{e_1,\dots,e_d\}$. Then for any positive
mass function $f:\cM\to[0,1]$,
$\lambda(\heatbath{\pi}{f}{\cF}{\cM})=1-\min(f)$. 
\end{prop}
\begin{proof}
It is easy to verify that $\cF^c(\cM)$ has the intersecting ray property
and thus Proposition~\ref{p:UpperSLEMBound} shows
$\lambda(\heatbath{\pi}{f}{\cF}{\cM})\le 1-\min(f)$. Assume that
$\min(f)=f(e_i)$. The connected
components of $\cF^c(\{e_1,\dots,e_d\}\setminus\{e_i\})$ are the layers $V_j:=\{u\in
\cF: u_i=j\}$ for any $j\in[n_i]$ and the rays through $\cF$ parallel are
$\cR_{k}:=\{(0,k)+s\cdot e_i: s\in[n_i]\}$ for
$k=(k_1,\dots,k_{i-1},k_{i+1},\dots,k_d)\in[n_1]\times\cdots\times[n_{i-1}]\times[n_{i+1}]\times\cdots\times[n_d]$. In particular, any
ray intersects any connected component exactly once. Thus, the matrix
$(|R_k\cap V_j|)_{k,j}$ is the all-ones matrix, which has a
non-trivial kernel. Proposition~\ref{p:KernelRayMat} implies 
$\lambda(\heatbath{\pi}{f}{\cF}{\cM})\ge 1-f(e_i)$.
\end{proof}

\begin{remark}\label{r:RooksWalk}
In the special case $n:=n_1=\dots=n_d$ and
$f:\{e_1,\dots,e_d\}\to[0,1]$ the uniform distribution in
Proposition~\ref{p:HeatBathOnHyperrectangle}, the heat-bath
random walk on $[n]^d$ is known as \emph{Rook's walk} in the
literature. In this case, Proposition~\ref{p:HeatBathOnHyperrectangle}
is exactly~\cite[Proposition~2.3]{Kim2012}.
In~\cite{Mcleman2015}, upper bounds on the mixing time of the Rook's
walk were obtained with \emph{path-coupling}.
\end{remark}

The stationary distribution of the heat-bath random walk is
independent of the actual mass function on the Markov moves. 
The problem of finding the mass function which leads to the fastest
mixing behaviour can be formulated as the following optimization
problem:
\begin{equation}\label{equ:SLEMOpti}
\arg\min\left\{\lambda(\heatbath{\pi}{f}{\cF}{\cM}): f:\cM\to(0,1),
\sum_{m\in\cM}f(m)=1\right\}.
\end{equation}
It follows from
Proposition~\ref{p:HeatBathOnHyperrectangle} that the optimal value of
\eqref{equ:SLEMOpti} for $\cF=[n_1]\times\cdots\times[n_d]$,
$\cM=\{e_1,\dots,e_d\}$, and the uniform distribution $\pi$ on $\cF$
is the uniform distribution on $\cM$. Another example where the
uniform distribution is the optimal solution to~\eqref{equ:SLEMOpti},
but where the verification is more involved, is presented in
Example~\ref{ex:SolvingSLEMOpti}.

\begin{example}\label{ex:SolvingSLEMOpti}
Let $\cF=[2]\times[5]$ as in Example~\ref{ex:AddingMovesSlowingChain}
and consider $\cM=\{e_1,2e_1+e_2\}$.
We investigate for which $\mu\in(0,1)$, the transition matrix
$\mu\heatbathmove{\pi}{\cF}{e_1}+(1-\mu)\heatbathmove{\pi}{\cF}{2e_1+e_2}$ has the
smallest second largest eigenvalue modulus. Its characteristic polynomial in $\QQ[\mu,x]$ is
$$-\frac{1}{25}x^4(x-1)(\mu+x-1)^6(-5x^2+5x+2\mu^2-2\mu)(-5x^2+5x+4\mu^2-4\mu)$$
and hence its eigenvalues are 
\begin{equation*}
\begin{split}
x_1(\mu):=1,& \quad x_2(\mu):=1-\mu,\\
x_3(\mu):=\frac{1}{2}\left[1+\sqrt{1+\frac{8}{5}(\mu^2-\mu)}\right],
&\quad
x_4(\mu):=\frac{1}{2}\left[1-\sqrt{1+\frac{8}{5}(\mu^2-\mu)}\right],\\
x_5(\mu):=\frac{1}{2}\left[1+\sqrt{1+4(\mu^2-\mu)}\right],&\quad
x_6(\mu):=\frac{1}{2}\left[1-\sqrt{1+4(\mu^2-\mu)}\right].\\
\end{split}
\end{equation*}
It is straightforward to check that $x_5(\mu)>\frac{1}{2}> x_6(\mu)$, $x_3(\mu)>\frac{1}{2}>
x_4(\mu)$. Since $\mu^2-\mu<0$ for $u\in(0,1)$ and 
$x_3(\mu)\ge x_6(\mu)$. We can show that
$x_4(\mu)\ge x_2(\mu)$ and thus
\begin{equation*}
\lambda(\mu\heatbathmove{\pi}{\cF}{e_1}+(1-\mu)\heatbathmove{\pi}{\cF}{2e_1+e_2})=\frac{1}{2}\left[1+\sqrt{1+\frac{8}{5}(\mu^2-\mu)}\right].
\end{equation*}
The fastest heat-bath random walk on $\cF(\cM)$ which converges to
uniform is thus obtained for
$\mu=\frac{1}{2}$, i.e. when the moves are selected uniformly. The
second largest eigenvalue in this case is
$\frac{1}{10}(5+\sqrt{15})\approx 0.887$, which is larger than the
second largest eigenvalue of the heat-bath walk that selects uniformly
from $\{e_1,e_2\}$ (see Proposition~\ref{p:HeatBathOnHyperrectangle}).
\end{example}

\section{Augmenting Markov bases}

It follows from our investigation in Section~\ref{s:Diameter} that the
diameter of all compressed fiber graphs coming from a fixed integer matrix
$A\in\ZZ^{m\times d}$ can be bounded from above by a constant.
However, Markov moves can be used twice in a minimal path which can
make the diameter of the compressed fiber graph larger than the size
of the Markov basis. The next definition puts more constraints on the
Markov basis and postulates the existence of a path that uses every
move from the Markov basis at most once.

\begin{defn}\label{d:Augmentation}
Let $\cF\subset\ZZ^d$ be a finite set and
$\cM=\{m_1,\dots,m_k\}\subset\ZZ^d$. An \emph{augmenting path} between
distinct $u,v\in\cF$ of length $r\in\NN$ is a path in $\cF^c(\cM)$ of
the form
\begin{equation*}
u\to u+\lambda_{i_1}m_{i_1}\to u+\lambda_{i_1}m_{i_1}+\lambda_{i_2}m_{i_2}
\to\cdots \to u+\sum_{k=1}^r\lambda_{i_k}m_{i_k}=v
\end{equation*}
with distinct indices $i_1,\dots,i_r\in[k]$. An augmenting path is
\emph{minimal} for $u,v\in\cF$ if there exists no shorter augmenting
path between $u$ and $v$ in $\cF^c(\cM)$. 
A Markov basis $\cM$ for $\cF$ is \emph{augmenting} if there is an
augmenting path between any distinct nodes in $\cF$. The
\emph{augmentation length} $\auglen{\cM}{\cF}$ of an augmenting Markov
basis $\cM$ is the maximum length of all minimal augmenting paths in
$\cF^c(\cM)$.
\end{defn}

Not every Markov basis is augmenting (see Example~\ref{ex:StairCase}),
but the diameter of compressed fiber graphs that use an augmenting
Markov basis is at most the number of the moves. For fiber graphs
coming from an integer matrix, an augmenting Markov basis for all of its
fibers can be computed (Remark~\ref{r:GraverAugmentation}).

\begin{remark}\label{r:GraverAugmentation}
Let $A\in\ZZ^{m\times d}$  with $\ker_\ZZ(A)\cap\NN^d=\{0\}$ and let
$b\in\cone{A}$. The
Graver basis is an augmenting Markov basis for
$\fiber{A}{b}$ for any $b\in\cone{A}$. We claim that when $A$ is
totally unimodular, then $\auglen{\graver{A}}{\fiber{A}{b}}\le d^2(\rank(A)+1)$. In
particular, the augmentation length is independent of the right-hand
side $b$. Let $u,v\in\fiber{A}{b}$ be arbitrary and for $i\in\NN$, let
$l_i:=\min\{u_i,v_i\}$, $w_i:=\max\{u_i,v_i\}$, and
$c_i:=\mathrm{sign}(u_i-v_i)\in\{-1,0,1\}$.
Then $v$ is the unique optimal value of the linear integer optimization
problem
\begin{equation*}
\min\{c^Tx: Ax=b, l\le x\le w,x\in\ZZ^d\}.
\end{equation*}
A \emph{discrete steepest decent} as defined in
\cite[Definition~3]{loera-augmentation} using Graver moves needs at most
$\|c\|_1\cdot d\cdot(\rank(A)+1)\le d^2\cdot(\rank(A)+1)$ many
augmentations from $u$ to reach the optimal value $v$. We refer to 
\cite[Corollary~8]{loera-augmentation} which ensures that every Graver move is used
at most once. Note that in~\cite{loera-augmentation}, $x$ is constrained to
$x\ge 0$ instead to $x\ge l$, but their argument works for any lower
bound.
\end{remark}

\begin{example}\label{ex:AugmentingMarkovBasis}
Fix $d\in\NN$ and consider $A$ and $\cM$ from
Example~\ref{ex:NonNormReducingNormLike}. We show that $\cM$ is an
augmenting Markov basis for $\fiber{A}{b}$ for any $b\in\NN$.
Let $u,v\in\fiber{A}{b}$ be distinct, then there exists $i\in[d]$ such
that $u_i>v_i$ or $u_i<v_i$, thus, we can walk from $u$ to
$u':=u+(u_i-v_i)(e_1-e_i)$ or from $v$ to $v':=v+(v_i-u_i)(e_1-e_i)$.
In any case, after that augmentation, the pairs $(u',v)$ and $(v',u)$
coincide in the $i$th coordinate and thus we find an augmenting path
by induction on the dimension $d$. We have used at most $d-1$ many
edges in these paths and hence $\auglen{\cM}{\fiber{A}{b}}\le d-1$ for
all $b\in\NN$.
\end{example}

We now show that the augmentation length is essentially bounded from
below by the dimension of the node set and hence the bound observed in
Example~\ref{ex:AugmentingMarkovBasis} cannot be improved. We first need the
following lemma.

\begin{lemma}\label{l:VectorSpace}
Let $v_1,\dots,v_k\in\QQ^d$ such that any
$v\in\spannQQ{v_1,\dots,v_k}$ can be represented by a linear
combination of $r$ vectors. Then
$\dim(\spannQQ{v_1,\dots,v_k})\le r$.
\end{lemma}
\begin{proof}
Let $\mathfrak{B}\subset\mathfrak{P}(v_1,\dots,v_k)$ the set of all
subsets of cardinality $r$. By our assumption,
$\cup_{B\in\mathfrak{B}}\spannQQ{B}=\spannQQ{v_1,\dots,v_k}$.
Since $\dim(\spannQQ{B})\le r$ for all $B\in\mathfrak{B}$ and
since $\mathfrak{B}$ is finite, the claim follows.
\end{proof}

\begin{prop}\label{p:AugmentingDilatation}
Let $\cP\subset\QQ^d$ be polytope and let $\cM\subset\ZZ^d$ be an
augmenting Markov basis for
$\cF_i:=(i\cdot\cP)\cap\ZZ^d$ for all $i\in\NN$. Then
$\dim(\cP)\le\max_{i\in\NN}\auglen{\cM}{\cF_i}$.
\end{prop}
\begin{proof}
Without restricting generality, we can assume that
$0\in\cP$. Let $V:=\spannQQ{\cP}$ be the $\QQ$-span of
$\cP$, then $\dim(\cP)=\dim(V)$. 
We must have $\dim(\spannQQ{\cM})=\dim(V)$ since 
$\dim(\cP)=\dim(\conv{\cF_i})$ for $i$ sufficiently large and since $\cM$ is a Markov basis
for $\cF_i$. Define $r:=\max_{i\in\NN}\auglen{\cM}{\cF_i}$
and choose any
non-zero $v\in V$ and $u\in\mathrm{relint}(\cP)\subset\QQ^d$. Then there exists
$\delta\in\QQ_{>0}$ such that $u+\delta v\in\cP$. Thus,
$\frac{1}{\delta}u+v\in\frac{1}{\delta}\cP$. Let $c\in\NN_{\ge 1}$
such that $i:=\frac{c}{\delta}\in\NN$ and 
$w:=\frac{c}{\delta}u\in\ZZ^d$. Then
$w+c
v=c(\frac{1}{\delta}u+v)\in(i\cdot\cP)\cap\ZZ^d=\cF_{i}$. By
assumption, there exists an augmenting path from $w$ to $w+c v$ using
only $r$ elements from $\cM$. Put differently, the element $c v$ from
$V$ can be represented by a linear combination of $r$ vectors from $\cM$.
Since $v$ was chosen arbitrarily, Lemma~\ref{l:VectorSpace} implies
$\dim(\cP)=\dim(V)\le r$.
\end{proof}

\begin{remark}
It is a consequence from Proposition~\ref{p:AugmentingDilatation}
that for any matrix $A\in\ZZ^{m\times d}$ with
$\ker_\ZZ(A)\cap\NN^d=\{0\}$ and an augmenting Markov
basis $\cM$, there exists $\cF\in\cP_A$ such that
$\auglen{\cM}{\cF}\ge\dim(\ker_\ZZ(A))$.
\end{remark}

Let us now shortly recall the framework from~\cite{Sinclair1992} which
is necessary to prove our main theorem. Let $G=(V,E)$ be a graph. For
any ordered pair of distinct nodes $(x,y)\in V\times V$, let
$p_{x,y}\subseteq E$ be a path from $x$ to $y$ in $G$ and let
$\Gamma:=\{p_{x,y}: (x,y)\in V\times V, x\neq y\}$ be the collection
of these paths, then $\Gamma$ is \emph{a set of canonical paths}.
Let for any edge $e\in E$, $\Gamma_e:=\{p\in\Gamma: e\in p\}$ be the
set of paths from $\Gamma$ that use $e$. Now, let $\cH:V\times V\to[0,1]$ be a symmetric
random walk on $G$ and define
\begin{equation*}
\rho(\Gamma,\cH):=\frac{\max\{|p|:
p\in\Gamma\}}{|V|}\cdot\max_{\{u,v\}\in
E}\frac{|\Gamma_{\{u,v\}}|}{\cH(u,v)}.
\end{equation*}
Observe that symmetry of $\cH$ is needed to make $\rho(\Gamma,\cH)$
well-defined. This can be used to prove the following upper bound on
the second largest eigenvalue.

\begin{lemma}\label{l:CanonicalPaths}
Let $G$ be a graph, $\cH$ be a symmetric random walk on $G$, and
$\Gamma$ be a set of canonical paths in $G$. Then
$\lambda_2(\cH)\le 1-\frac{1}{\rho(\Gamma,\cH)}$.
\end{lemma}
\begin{proof}
The stationary distribution of $\cH$ is the uniform distribution and
thus the statement is a direct consequence of~\cite[Theorem~5]{Sinclair1992},
since $\rho(\Gamma,\cH)$ is an upper bound on the constant defined
in~\cite[equation~4]{Sinclair1992}.
\end{proof}

\begin{thm}\label{t:MixingofAugmentingMarkovBases}
Let $\cF\subset\ZZ^d$ be finite and let
$\cM:=\{m_1,\dots,m_k\}\subset\ZZ^d$ be an augmenting Markov
basis. Let $\pi$ be the uniform and $f$ be a positive
distribution on $\cF$ and $\cM$ respectively. For $i\in[k]$, let
$r_i:=\max\{|\polyray{u}{m_i}{\cF}|:u\in\cF\}$ 
and suppose that $r_1\ge r_2\ge\dots\ge r_k$. Then
\begin{equation*}
\lambda(\heatbath{\pi}{f}{\cM}{\cF})\le1-\frac{|\cF|\cdot\min(f)}{\auglen{\cM}{\cF}\cdot\auglen{\cM}{\cF}!\cdot
3^{\auglen{\cM}{\cF}-1}\cdot2^{|\cM|}\cdot r_1r_2\cdots
r_{\auglen{\cM}{\cF}}}.
\end{equation*}
\end{thm}
\begin{proof}
Choose for any distinct
$u,v\in\cF$ an augmenting path $p_{u,v}$ of minimal length in
$\cF^c(\cM)$ and let $\Gamma$ be the collection of all
these paths. Let $u+\mu m_k=v$ be an edge in
$\cF^c(\cM)$, then our goal is to bound $|\Gamma_{\{u,v\}}|$ from
above. Let $\mathcal{S}:=\{S\subseteq[r]: |S|\le\auglen{\cM}{\cF},
k\in S\}$ and take any path $p_{x,y}\in\Gamma_{\{u,v\}}$. Then there exists
$S:=\{i_1,\dots,i_s\}$ with $s:=|S|\le\auglen{\cM}{\cF}$ such that
$x+\sum_{k=1}^s\lambda_{i_k}m_{i_k}=y$. Since $p_{x,y}$ uses the edge
$\{u,v\}$, there is $j\in[s]$ such that $i_j=k$ and
$\lambda_{i_j}=\mu$. Since $|\lambda_{i_k}|\le r_{i_k}$, there are at
most  
$$s!\cdot (2r_{i_1}+1)\cdots (2r_{i_{j-1}}+1)\cdot(2r_{i_{j+1}}+1)\cdots
(2r_{i_s}+1)\le
s!\cdot 3^{s-1}\prod_{t\in S\setminus\{k\}}r_{t}$$
paths in
$\Gamma_{\{u,v\}}$ that uses the edge $\{u,v\}$ and the moves
$m_{i_1},\dots,m_{i_{j-1}},m_{i_{j+1}}\dots,m_{i_s}$. Since all the
paths are minimal, they have length at most $\auglen{\cM}{\cF}$ so
indeed every path in $\Gamma$ has that form.
\begin{equation*}
\frac{|\Gamma_{u,v}|}{\heatbath{\pi}{f}{\cF}{\cM}(u,v)}\le
3^{\auglen{\cM}{\cF}-1}\frac{\sum_{S\in\mathcal{S}}\left(|S|!\prod_{t\in
S\setminus\{k\}}r_{t}\right)}{f(m_{i_j})\cdot\frac{1}{|\ray{u}{m_{i_j}}|}}\le\frac{3^{\auglen{\cM}{\cF}-1}\cdot
\auglen{\cM}{\cF}!\cdot|\mathcal{S}|\cdot
r_1r_2\dots r_{\auglen{\cM}{\cF}}}{f(m_{i_j})},
\end{equation*}
where we have used the assumption $r_1\ge r_2\ge \dots\ge r_k$.
Bounding $|\mathcal{S}|$ rigorously from above by $2^{|\cM|}$, the claim follows
from Lemma~\ref{l:CanonicalPaths}.
\end{proof}

\begin{defn}
Let $\cF\subset\ZZ^d$ and $\cM\subset\ZZ^d$ be finite sets. The longest
ray through $\cF$ along vectors of $\cM$ is
$\longestRay{\cF}{\cM}:=\arg\max\{|\polyray{u}{m}{\cF}|: m\in\cM, u\in
\cF\}$.
\end{defn}

\begin{cor}\label{c:AugmentingExpander}
Let $(\cF_i)_{i\in\NN}$ be a sequence of finite sets in $\ZZ^d$ and let
$\pi_i$ be the uniform distribution on $\cF_i$. Let $\cM\subset\ZZ^d$ be
an augmenting Markov basis for $\cF_i$ with $\auglen{\cM}{\cF_i}\le\dim(\cF_i)$
and suppose that
$(|\longestRay{\cF_i}{\cM}|)^{\dim(\cF_i)})_{i\in\NN}\in\mathcal{O}(|\cF_i|)_{i\in\NN}$.
Then for any positive mass function $f:\cM\to[0,1]$, there exists $\epsilon>0$ such that
$\lambda(\heatbath{\pi_i}{f}{\cF_i}{\cM})\le 1-\epsilon$ for all
$i\in\NN$.
\end{cor}
\begin{proof}
This is a straightforward application of
Theorem~\ref{t:MixingofAugmentingMarkovBases}. 
\end{proof}

\begin{cor}\label{c:ExpanderDilatation}
Let $\cP\subset\ZZ^d$ be a polytope, $\cF_i:=(i\cdot\cP)\cap\ZZ^d$
for $i\in\NN$, and let $\pi_i$ be the uniform distribution on $\cF_i$.
Suppose that $\cM\subset\ZZ^d$ is an augmenting Markov basis
$\{\cF_i:i\in\NN\}$ such that $\auglen{\cM}{\cF_i}\le\dim(\cP)$ for
all $i\in\NN$. 
Then for any positive mass function $f:\cM\to[0,1]$, there exists
$\epsilon>0$ such that $\lambda(\heatbath{\pi_i}{f}{\cF_i}{\cM})\le
1-\epsilon$ for all $i\in\NN$.
\end{cor}
\begin{proof}
Let $r:=\dim(\cP)$. We first show that
$(|\longestRay{\cF_i}{\cM}|)_{i\in\NN}\in\mathcal{O}(i)_{i\in\NN}$.
Write $\cM=\{m_1,\dots,m_k\}$ and denote by
$l_i:=\max\{|(u+m_i\cdot\ZZ)\cap\cP|: u\in\cP\}$ be the length of the
longest ray through the polytope $\cP$ along $m_i$.  It suffices to
prove that $i\cdot(l_k+1)$ is an upper bound on the length of any ray
along $m_k$ through $\cF_i$. For that, let $u\in\cF_i$ such that
$u+\lambda m_k\in\cF_i$ for some $\lambda\in\NN$, then
$\frac{1}{i}u+\frac{\lambda}{i} m_k\in\cP$ and thus
$\lfloor\frac{\lambda}{i}\rfloor\le l_k$, which gives $\lambda\le
i\cdot(l_k+1)$. With $C:=\max\{l_1,\dots,l_k\}+1$ we have
$|\longestRay{\cF_i}{\cM}|\le C\cdot i$.
Ehrhart's theorem~\cite[Theorem~3.23]{Beck2007} gives
$(|\cF_i|)_{i\in\NN}\in\Omega(i^r)_{i\in\NN}$ and
since $|\longestRay{\cF_i}{\cM}|\le C\cdot i$, we have
$(|\longestRay{\cF_i}{\cM}|^r)_{i\in\NN}\in\mathcal{O}(|\cF_i|)_{i\in\NN}$.
An application of Corollary~\ref{c:AugmentingExpander} proves the claim.
\end{proof}

\begin{example}\label{e:CrossPoly}
Fix $d,r\in\NN$ and let $\cC_{d,r}:=\{u\in\ZZ^d: \|u\|_1\le r\}$ be
the set of integers of the $d$-dimensional cross-polytope with radius
$r$. The set $\cM_d=\{e_1,\dots,e_d\}$
is a Markov basis for
$\cC_{d,r}$ for any $r\in\NN$. We show that $\cM_d$ is an augmenting
Markov basis whose augmentation length is at most $d$. For that, let 
$u,v\in\cC_{d,r}$ distinct elements. We claim that there exists
$i\in[d]$ such that $x_i\neq v_i$ and $u_i+(v_i-u_i)\in\cC_{d,r}$.
Let $S\subseteq[d]$ be the set of indices where $u$ and $v$ differ and let
$s= r- ||u||_1$. If $|S|=1$, then the result is clear so
suppose $|S| \geq 2$. If the result doesn't hold then for all
$i\in S$, $|v_i|-|u_i| > s$. It follows that 
\begin{equation*} 
\|v\|_1 = \sum_{i \notin S} |u_i| + \sum_{i \in S} |v_i|
 > \sum_{i \notin S} |u_i| + \sum_{i \in S} s + |u_i| = |S_{uv}|\cdot s + \|u\|_1
 = (|S|-1) \cdot s + r.
\end{equation*}
But we assumed that $v \in\cC_{d,r}$. It follows that for any pair of
points $u,v$ in $\cC_{d,r}$, there is a walk, using the unit vectors
as moves, that uses each move at most once.
Corollary~\ref{c:AugmentingExpander} yield that for any $d\in\NN$, the
second largest eigenvalue modulus of the heat-bath random walk on
$\cC_{d,r}$ with uniform as stationary distribution can be strictly
bounded away from $1$ for $r\to\infty$.
\end{example}

The bound on the second largest eigenvalue in
Theorem~\ref{t:MixingofAugmentingMarkovBases} is quite general and can
be improved vastly, provided one has better control over the paths.
For example, this can be achieved for hyperrectangles intersected with
a halfspace.

\begin{prop}\label{p:Hyperrectangles}
Let $a\in\NN^d_{>0}$, $b\in\NN$,
$\cF=\{u\in\NN^d: a^T\cdot u\le b\}$, and
$\cM:=\{e_1,\dots,e_d\}$. If $\pi$ and $f$ are the uniform
distributions on $\cF$ and $\cM$ respectively, then
\begin{equation*}
\lambda(\heatbath{\pi}{f}{\cF}{\cM})\le
1-\frac{|\cF|}{d^2}\prod_{i=1}^d\frac{a_i}{b}.
\end{equation*}
\end{prop}
\begin{proof}
Observe that $\cM$ is a Markov basis for $\cF$ since all nodes are connected
with $0\in\cF$. Let $u,v\in\cF$ be distinct. We first show
that there exists $k\in[d]$ such that $u_k\neq v_k$ and
$u+(v_k-u_k)e_k\in\cF$. If $u\le v$, the statement trivially holds.
Otherwise, there exists $k\in [d]$ such that $u_k>v_k$ and the vector
obtained by replacing the $k$th coordinate of $u$ by $v_k$ remains in $\cF$.
Now, consider for the following path between $u$ and $v$: Choose the smallest index
$k\in[d]$ such that $u_k\neq v_k$ and such that $u+(v_k-u_k)\cdot
e_k\in\cF$ and proceed recursively with $u+(v_k-u_k)$ and $v$. This
gives a path $p_{u,v}$ between $u$ and $v$ of length at most $d$. Let
$\Gamma$ be the collection of all these paths.
We want to apply Lemma~\ref{l:CanonicalPaths}. Thus, let $x\in\cF$ and
consider the edge $x\rightarrow x+c\cdot e_s$. Let us count the paths
$p_{u,v}$ that use that edge. Let $u,v\in\cF$ and let $k_1,\dots,k_r\in
[d]$ be distinct indices such that
\begin{equation*}
u\rightarrow u+(v_{k_1}-u_{k_1})e_{k_1}\rightarrow
u+(v_{k_1}-u_{k_1})e_{k_1}+(v_{k_2}-u_{k_2})e_{k_2}\rightarrow\cdots\rightarrow
v
\end{equation*}
represents the path $p_{u,v}$ constructed by the upper rule. Assume
that $p_{u,v}$ uses the edge $\{x,x+ce_s\}$ and let $k_l=s$ and
$(v_{k_l}-u_{k_l})=c$. In particular,
\begin{equation*}
\begin{split}
u&+(v_{k_1}-u_{k_1})e_{k_1}+\cdots +(v_{k_{l-1}}-u_{k_{l-1}})e_{k_{l-1}}=x\\
x&+(v_{k_l}-u_{k_l})e_{k_l}+\cdots +(v_{k_r}-u_{k_r})e_{k_r}=v.
\end{split}
\end{equation*}
We see that $v_{k_t}=x_{k_t}$ for all $t< l$ and that
$u_{k_t}=x_{k_t}$ 
for all $t\ge l$. In particular, $v_{k_l}=u_{k_l}+c=x_{k_l}+c$ is also
fixed. The coordinates $u_{k_t}$ and $v_{k_t}$ are bounded
from above by $\frac{b}{a_{k_t}}$ for
all $t\in[r]$, and hence there can be at most 
\begin{equation*}
\left(\prod_{t=1}^{l-1}\frac{b}{a_{k_t}}\right)\cdot
\left(\prod_{t=l+1}^{r}\frac{b}{a_{k_t}}\right).
\end{equation*}
Since $k_1,\dots,k_t$ are distinct coordinate indices, we have
\begin{equation*}
\frac{|\Gamma_{x,x+c\cdot e_s}|}{\heatbath{\pi}{f}{\cF}{\cM}(x,x+c\cdot
e_s)}\le d\cdot\prod_{i=1}^d\frac{b}{a_i}.
\end{equation*}
Lemma~\ref{l:CanonicalPaths} finishes the proof.
\end{proof}

In fixed dimension, Proposition~\ref{p:Hyperrectangles} leads to rapid
mixing, but for $d\to\infty$, no statement can be made.
In~\cite{Morris2004}, it was shown that the simple walk with an
additional halting probability on $\{u\in\NN^d: a^tu\le
b\}\cap\{0,1\}^d$ has mixing time in $\mathcal{O}(d^{4.5+\epsilon})$.
For zero-one polytopes, simple and heat-bath walk coincide and we are
confident that a similar statement holds without the restriction on
zero-one polytopes.

The heat-bath random walk mixes rapidly when an augmenting Markov
basis with a small augmentation length is used. We think that it is
interesting to question how might an augmenting Markov bases be obtained and
how their augmentation length can be improved.

\begin{question}
Let $\cM$ be an augmenting Markov basis of $A$.  Can we find finitely
many moves $m_1,\dots,m_k$ such that the augmentation length of
$\cM\cup\{m_1,\dots,m_k\}$ on $\fiber{A}{b}$ is at most
$\dim(\ker_\ZZ(A))$ for all $b\in\cone{A}$?
\end{question}

\bibliographystyle{amsplain}
\bibliography{heatbath}
\end{document}